\documentclass[10pt]{article}

\usepackage{amsmath,amsfonts,amssymb}
\usepackage[all]{xy}
\usepackage{pst-all}
\usepackage{theorem}

\newtheorem{proposition}{Proposition}
\newtheorem{theorem}{Theorem}
\newtheorem{corollary}{Corollary}
\newtheorem{lemma}{Lemma}
\newtheorem{Definition}{Definition}

{ \theorembodyfont{\rmfamily}
\newtheorem{remark}{Remark}[section]
\newtheorem{example}{Example}[section]
}

\newenvironment{proof}
  {\begin{trivlist}
  \item[\textit{\quad\quad\textsc{ Proof.}}]}
  {$\square$\end{trivlist}}

\title{Local and Global Aspects of Lie's
Superposition Theorem}
\author{David Bl\'azquez-Sanz \& Juan Jos\'e Morales-Ruiz\footnote{This research has been partially
 supported by grant MCyT-FEDER MTM2006-00478of Spanish goverment, and the Sergio Arboleda University 
Research Agency CIVILIZAR.}}

\begin{document}
\maketitle

\begin{abstract}
  In this paper we give the global conditions for an ordinary 
differential equation to admit a superposition law of solutions
in the classical sense. This completes the well-known Lie superposition
theorem. We introduce rigorous notions of pretransitive Lie group action
and Lie-Vessiot systems. We proof that pretransitive Lie group actions are
transitive. We proof that an ordinary differential equation 
admit a superposition law if and only if it is a pretransitive Lie-Vessiot 
system. It means that its enveloping algebra is spanned by 
fundamental fields of a pretransitive Lie group action. We discuss the relationship
of superposition laws with differential Galois theory and review
the classical result of Lie. 
\end{abstract}


\section{Introduction and Main Results}

Superposition laws of solutions of differential equations
attracted the attention of researchers since the statement
of the question by S. Lie in the 19th century. They
appear widely, sometimes as reflect of certain linear phenomena,
and sometimes as genuine non-linear superposition (see Example \ref{ExampleWeierstrass}). 
The local characterization of ordinary differential equations admitting 
superposition laws was partially done by S. Lie and G. Sheffers 
in \cite{Lie1885}. In the last quarter of 20th century, superposition
laws attracted again the attention of pure and applied mathematicians.
Shnider and Winternitz, in \cite{Winternitz1984}, classify the germs
of low dimension differential equations admitting superposition laws. This work
can be seen as a classification of germs of homogeneous spaces. There
is also a number of applied papers on the subject, including 
\cite{Winternitz1981, Winternitz1985, Winternitz1987, Winternitz1996,
Winternitz1999, CGR2001, CR2002}. A good treatment of superposition
laws can be seen in the textbook \cite{Ibragimov1999}, and there is
a monography completely devoted to the subject \cite{CGM2000}.

However, the theoretical approach to superposition laws 
relies is Lie's characterization. This is a local result, and avoid the 
global problem of construction of group. This problem
was tackled first by Vessiot in \cite{Vessiot1894}, who analyzed the same 
subject that Lie, but under a slighly different point of view, more
related with differential Galois theory. Instead of taking the Lie
algebra of the group from the joint invariants, Vessiot extract the
group structure directly from the superposition law. This idea is 
systematically used along Section \ref{SectionGlobal}. Those are
the key points that allow is to relate superposition of solutions
with differential Galois theory. 

 In the original infinitesimal Lie's approach, it is necessary to
integrate a Lie algebra of autonomous vector fields, in order to obtain 
a number of arbitrary constants that correspond to joint invariants 
-- by joint invariants we mean first integral of the differential equation lifted to 
the cartesian power. But, in the general case, there are not
appropiated joint invariants (see Example \ref{exampleLorentz}). 
Instead of this, we recover at least a foliation
on the cartesian power of the phase space, and each leave of the foliation
correspond to a determination of those theoretical arbitrary constants.
This defect was noticed by Cari\~nena, Grabowsky and Marmo \cite{CGM2007}.
They substitute the notion of superposition law for the milder notion of
local superposition; then the classical statement of Lie holds. 
Here we give the conditions for the existence of a superposition
law in the global classical sense, that is, a globally defined map that
allow us to recover the general solution.

  An ordinary differential equation admitting a \emph{fundamental
system of solutions} is, by definition, a system of non-autonomous
differential equations,
\begin{equation}\label{ATeq1}
  \frac{dx_i}{dt} = F_i(t,x_1,\dots,x_n) \quad i=1,\ldots,n
\end{equation}
for which there exists a set of formulae,
\begin{equation}\label{ATeq2}
\varphi_i(x^{(1)},\ldots,x^{(r)};\lambda_1,\ldots,\lambda_n) \quad
i=1,\ldots,n
\end{equation}
expressing the \emph{general solution} as function of $r$ particular
solutions of \eqref{ATeq1} and some arbitrary constants $\lambda_i$.
This means that for $r$ particular solutions \linebreak 
$(x_1^{(i)}(t),\ldots,x^{(i)}_n(t))$ of the
equations, satisfying certain non-degeneracy condition,
\begin{equation}\label{ATeq3}
  x_i(t) = \varphi_i(x^{(1)}_1(t),\ldots, x^{(1)}_n(t),\ldots,x^{(r)}_n(t);\lambda_1,\ldots,\lambda_n)
\end{equation}
is the general solution of the equation \eqref{ATeq1}. In
\cite{Lie1893} Lie also stated that the arbitrary constants
$\lambda_i$ parametrize the solution space, in the sense that for
different constants, we obtain different solutions: there are not
functional relations between the arbitrary constants $\lambda_i$.
The set of formulae $\varphi_i$ is usually referred to as a
\emph{superposition law for solutions} of \eqref{ATeq1}.

  The class of ordinary differential equations admitting \emph{fundamental
systems of solutions}, or \emph{superposition laws}, was introduced by S. Lie 
in 1885 \cite{Lie1885}, as certain class of auxiliary equations appearing 
in his integration methods for ordinary differential equations. Further 
development is due to Guldberg and mainly to Vessiot \cite{Vessiot1894}. 
The characterization of such class of differential equations is given by 
S. Lie and G. Scheffers in 1893 \cite{Lie1893}, and it is know as 
\emph{Lie's superposition theorem}, or \emph{Lie-Scheffers theorem}. There
is a Lie algebra canonically attached to any non-autonomous differential 
equation, the so called \emph{Lie-Vessiot-Guldberg} or \emph{enveloping} algebra. 
Lie superposition theorem states that a differential equation admit a 
superposition law if and only if its enveloping algebra is
finite dimensional. This result, as observed by some contemporary authors 
(see \cite{CGM2000} and \cite{CGM2007}), is not true in the general case. 
In fact, Lie theorem characterizes a bigger class of differential 
equations, whichever admit certain invariant foliation known as 
\emph{local superposition law}. Here we find the characterization of
differential equations admitting superposition laws in the classical sense. 
In order to do so, we introduce the concepts of \emph{Lie-Vessiot systems} and
\emph{pretransitive Lie group actions}. Finally we obtain the whole picture
of global superposition laws.

\medskip

{\bf Theorem \ref{T1} (Global Lie).} 
\emph{A non-autonomous complex analytic vector field $\vec X$
in a complex analytic manifold $M$ admits a superposition law if and only if
it is a Lie-Vessiot system related to a pretransitive Lie group action
in $M$.}

\medskip
  Using this global philosophy it is easy to jump into the
algebraic category. First, is follows easly that rational (\emph{id est}, given
by rational functions of the coordinates)
superposition laws lead to rational actions of algebraic groups.

\medskip
{\bf Theorem \ref{T2}} 
\emph{Let $\vec X$ be a non-autonomous meromorphic vector field
in an algebraic manifold $M$ that admits a rational superposition law. Then,
it is a Lie-Vessiot system related to an algebraic action of an algebraic
group $G$ if $M$.}

\medskip
This result allow us to connect Lie's result with the frame of differential 
Galois theory. The main purpose of differential Galois theory is to classify
differential equations with respect to their integrability and reducibility,
and of course, solve those that are integrable. The first step in differential
Galois theory is the theory due to Picard and Vessiot that deals with
linear differential equations. In this frame, there is a correspondence
betweend linear differential equations and their differential Galois groups, 
which are algebraic groups of matrixes. Being the Lie-Vessiot systems a natural
generalization of linear equations, then it is expectable to generalize
the differential Galois theory for them. In \cite{Kolchin1973}, E. Kolchin
gives the theory of \emph{strongly normal extensions},
a  generalization of Picard-Vessiot theory in which the Galois group
are arbitrary algebraic groups. However, this theory deals with differential
field extensions and not with differential equations. Thus, the connection
with superposition laws remained hidden for long time. This connection was
first noticed by J. Kovacic, and exposed in a series of lectures 
(including his conference in Algebraic Methods in Dynamical Systems, Barcelona 2008). 
Here we give this connection in a explicit way.

\medskip

{\bf Theorem \ref{T3}} 
\emph{Let $\vec X$ be a non-autonomous meromorphic vector field in an algebraic manifold $M$
with coefficients in the Riemann surface $S$ that admits a rational superposition law.
Let $L$ be the differenital field spanned by the coordinates of any fundamental 
system of solutions of $\vec X$. Then, $\mathcal M(S) \subset L$ is a strongly 
normal extension in the sense of Kolchin.
Moreover, if $G$ is the algebraic group of transformations in $M$ induced by the 
superposition law, each particular solution $\sigma(t)$ of the associated automorphic system $\vec A$in $G$ induce an
injective map,
$${\rm Aut}_{\mathcal M(S)}(L) \to G.$$
which is an anti-morphism of groups.}

%
%
%

\section{Non-autonomous Complex Analytic Vector Fields}

{\bf Notation.} From now on, the phase space $M$ is a complex analytic
manifold,  $S$ is a Riemann surface together with a
holomorphic derivation $\partial \colon \mathcal O_S \to \mathcal O_S$,
beig $\mathcal O_S$ the sheaf of holomorphic functions in $S$.  By
\emph{the extended phase space} we mean the cartesian power $S\times M$.
We shall denote $t$ for the general point of $S$ and $x$ for the
general point of $M$. We write $\bar x$ for a $r$-frame
$(x^{(1)},\ldots,x^{(r)})\in M^r$. Under this rule, we shall write
$f(x)$ for functions in $M$, $f(t)$ for functions in $S$ and
$f(t,x)$ for functions in the extended phase space. Whenever we need
we take a local system of coordinates $x_1,\ldots,x_n$ for $M$. The induced
system of coordinates in $M^r$, component by components is then
$x_1^{(1)},\ldots,x_n^{(1)},\ldots,x_1^{(r)},\ldots, x_n^{(r)}$. 
The general point of the cartesian power $M^{r+1}$
is seen as a pair $(\bar x, x)$ of an $r$-frame and an point of $M$. Thus, local
coordinates in $M^{r+1}$ are $x_1^{(1)},\ldots,x_n^{(1)},\ldots,x_1^{(r)},\ldots, x_n^{(r)},
x_1,\ldots,x_n.$

\begin{Definition}
  A non-autonomous complex analytic vector field $\vec X$ in $M$, depending on
  the Riemann surface $S$, is an autonomous vector field in $S\times M$ compatible with $\partial$
  in the following sense:
  $$\vec X f(t) = \partial f(t)\quad\quad
  \xymatrix{\mathcal O_{S\times M} \ar[r]^-{\vec X} & \mathcal O_{S\times M} \\
  \mathcal O_S \ar[r]^-{\partial}\ar[u] & \mathcal O_S \ar[u]}$$
\end{Definition}

  In each cartesian power $M^r$ of $M$ we consider the \emph{lifted} vector
field $\vec X^r$. This is just the direct sum copies of $\vec X$
acting in each component of the cartesian power $M^r$. We have the
local expression for $\vec X$,
$$\vec X = \partial + \sum_{i=1}^n F_i(t,x) \frac{\partial}{\partial
x_i},$$ and also the local expression for $\vec X^r$, which is a
non-autonomous vector field in $M^r$,
\begin{equation}\label{EQlocalX}
\vec X^r = \partial + \sum_{i=1}^n
F_i(t,x^{(1)})\frac{\partial}{\partial x_i^{(1)}} + \ldots +
\sum_{i=1}^n F_i(t,x^{(r)})\frac{\partial}{\partial x_i^{(r)}}.
\end{equation}

  Let us consider $\vec X$ a non-autonomous vector field in $M$.
We can see $\vec X$ as an holomorphic map $\vec X \colon S \to
\mathfrak X(M)$, which assigns to each $t_0\in S$ an autonomous
vector field $\vec X_{t_0}$.

\begin{Definition}
  The \emph{enveloping} algebra of $\vec X$ is the Lie algebra of
  vector fields in $M$ spanned by the set vector fields $\{\vec
  X_{t_0}\}_{t_0\in S}$. The enveloping algebra of $\vec
  X$ is denoted $\mathfrak g(\vec X)$.
\end{Definition}

\begin{remark}
  Assume that there exist $\vec X_1,\ldots, \vec X_s$ autonomous
vector fields in $M$ spanning a finite dimensional Lie algebra,
and holomorphic functions \linebreak $f_1(t),\ldots, f_s(t)$ in $S$ such
that,
\begin{equation}\label{LieOriginalCondition}
\vec X = \partial + \sum_{i=1}^s f_i(t)\vec X_i, \quad\quad
 [\vec X_i,\vec X_j] = \sum_k c_{ij}^k \vec X_k.
\end{equation}
It follows that the enveloping algebra of $\vec X$ is a subalgebra of the one spanned
by $\vec X_1,\ldots, \vec X_s$. Reciprocally, if the enveloping algebra of 
$\vec X$ is of finite dimension, let us consider a basis $\{\vec X_i\}$. 
The coordinates of $\vec X_t$ in such basis, when $t$ varies in $S$, define 
holomorphic functions $f_i(t)$. Thus, we obtain an expression as above 
for $\vec X$. Expression \eqref{LieOriginalCondition} is the condition 
expressed in \cite{Lie1893}  by S. Lie, and it is clear that it is equivallent to the finite
dimension of the enveloping algebra.
\end{remark}

\section{Superposition Laws}

\begin{Definition}
  A superposition law for $\vec X$ is an analytic map
  $$\varphi\colon U \times P \to M,$$
  where $U$ is analytic open subset of $M^r$ and $P$ is an
  $n$-dimensional manifold, verifying:
  \begin{itemize}
  \item[(a)] $U$ is union of integral curves of $\vec X^r$
  \item[(b)] If $\bar x(t)$ is a solution of $\vec X^r$, defined for $t$
  in some open subset $S'\subset S$, then $x_{\lambda}(t) =\varphi(\bar
  x(t); \lambda)$, is the general solution of $\vec X$ for $t$ varying in
  $S'$.
  \end{itemize}
\end{Definition}

  The problem of dependence on parameters can be reduced to
the dependence on initial conditions. The parameter space is then
substituted by the phase space. In virtue of the following proposition, from
now on we will assume that the space of parameters of any superposition law
is the phase space itself.  

\begin{proposition}
  If $\vec X$ admits a superposition law, 
$$\phi\colon U\times P \to M, \quad\quad U \subset M^r,$$
  then it admits another superposition law
$$\varphi\colon U\times M \to M, \quad\quad U\subset M^r,$$
whose space of parameters is the space of the
initial conditions in the following sense: for a given 
$t_0\in S$ there is a solution $\bar x(t)$ defined in a
neighborhood of $t_0$ such that for all $x\in M$,
$\varphi(\bar x(t_0); x) = x.$
\end{proposition}

\begin{proof}
  Consider a superposition law $\phi$ as in the statement. 
  Let us choose $t_0\in S$ and
  certain local solution $\bar x_0(t)$ of $\vec X^r$. Then,
  $$x(t;\lambda) = \phi(\bar x_0(t);\lambda),$$ is the general
  solution. Define:
  $$\xi\colon P \to M,\quad \lambda\mapsto x(t_0;\lambda),$$
  then, by local existence and uniqueness of solutions for
  differential equations, for each $x_0\in M$ there exist a unique local
  solution $x(t)$ such that $x(t_0) = x_0$. Hence, there exist a map
  $\varphi$,
  $$\xymatrix{U\times P  \ar[r]^-{\phi} \ar[d]_-{Id\times \xi} & M \\ U \times M \ar[ru]_-{\varphi}}$$
  which is a superposition law for $\vec X$ satisfying the assumptions of the statement.
\end{proof}

\begin{example}[Linear systems]
Let us consider a linear system of ordinary differential equations,
$$\frac{dx_i}{dt} = \sum_{j=1}^n a_{ij}(t)x_j, \quad i=1,\ldots,n$$
as we should know, a linear combination of solutions of this system
is also a solution. Thus, the solution of the system is a $n$
dimensional vector space, and we can express the global solution as
linear combinations of $n$ linearly independent solutions. The
superposition law is written down as follows,
$$\mathbb C^{n\times n}_{x_i^{(j)}} \times \mathbb C^n_{\lambda_j} \to \mathbb
C^n, \quad (x_i^{(j)},\lambda_j) \mapsto (y_i) \quad y_i =
\sum_{j=1}^n \lambda_jx_i^{(j)}.$$
\end{example}

\begin{example}[Riccati equations]
Let us consider the ordinary differential equation,
$$\frac{dx}{dt} = a(t) + b(t)x + c(t)x^2,$$
let us consider four different solutions
$x_1(t),x_2(t),x_3(t),x_4(t)$. A direct computation gives that the
anharmonic ratio is constant,
$$\frac{d}{dt} \frac{(x_1 - x_2)(x_3-x_4)}{(x_1-x_4)(x_3-x_2)} =
0.$$If $x_1, x_2, x_3$ represent three known solutions, we can
obtain the unknown solution $x$ from the expression,
$$\lambda = \frac{(x_1 - x_2)(x_3-x)}{(x_1-x)(x_3-x_2)}$$
obtaining,
$$x = \frac{x_3(x_1-x_2)-\lambda
x_1(x_3-x_2)}{(x_1-x_2)-\lambda(x_3-x_2)}$$ which is the general
solution in function of the constant $\lambda$, and then a
superposition law for the Riccati equation.
\end{example}

Let us consider $\mathcal O_M$ the sheaf of holomorphic functions in
$M$, and for each $r$, the sheaf $\mathcal O_{M^r}$ of holomorphic
functions in the cartesian power $M^r$. We denote by  $\mathcal
O_{M^r}^{\vec X^r}$ the subsheaf of first integrals of the lifted
vector field $\vec X^r$ \eqref{EQlocalX} in $\mathcal O_{M^r}$. 
For the notion of \emph{regular ring} we follow the 
definition of \cite{Mun1982}.

\begin{Definition}
  Let $x$ be a point of $M$ and $\mathcal O_{M,x}$ be the ring of
germs of complex analytic functions in $x$. A subring $\mathcal
R\subset \mathcal O_{M,x}$ is a regular ring if it is
the ring of first integrals of $k$ germs in $x$ of vector fields
$\vec Y_1$,$\ldots$,$\vec Y_k$, which are $\mathbb C$-linearly independent at
$x$ and in involution: $$[\vec Y_i,\vec Y_j] = 0.$$ The
dimension $\dim \mathcal R$ is the number $\dim_{\mathbb C} M - k$.
\end{Definition}

Consider a sheaf of rings $\Psi\subset \mathcal O_{M^{r+1}}$ of
complex analytic functions in the cartesian power $M^{r+1}$. We say
that $\Psi$ is a sheaf of \emph{generically regular rings} if for a
generic point (\emph{i.e.} outside a closed analytic set)
$(\bar x,x) \in M^{r+1}$ the stalk $\Psi_{(\bar x,x)}$ is a
regular ring. A sheaf $\Psi$ of generically regular rings is the
sheaf of first integrals of a generically regular Frobenius
integrable foliation. We denote this foliation by $\mathcal
F_{\Psi}$.

\begin{Definition}\label{ATdef6}
A \emph{local} superposition law for $\vec X$ is a sheaf of rings  
$\Psi \subset \mathcal O_{M^{r+1}}$, for some
$r\in\mathbb N$, verifying:
\begin{itemize}
\item[(1)] $\Psi$ is a sheaf of generically regular rings of dimension $\geq n$.
\item[(2)] $\vec X^{r+1}\Psi = 0$, or equivalently, $\vec X^{r+1}$ is tangent
to $\mathcal F_{\Psi}$.
\item[(3)] $\mathcal F_{\Psi}$ is generically transversal to the fibers of the projection
\linebreak $M^{r+1}\to M^r$.
\end{itemize}
\end{Definition}

 This notion is found for first time in \cite{CGM2007}, in the language of
foliations. They prove that the existence of a local superposition
law of $\vec X$ is equivalent to finite dimensional
enveloping algebra, as we also are going to see.

\begin{proposition}
If $\vec X$ admits a superposition law, then it admits a local
superposition law.
\end{proposition}

\begin{proof}
Let us consider a superposition law for $\vec X$,
$$\varphi\colon U\times M \to M, \quad\quad U \subset M^r, \quad\quad \varphi =(\varphi_1,\ldots,\varphi_n).$$
We write the general solution,
\begin{equation}\label{ATeq20}
x(t) = \varphi(x^{(1)}(t),\ldots,x^{(r)}(t);\lambda).
\end{equation}
The local analytic dependency with respect to initial conditions
ensures that the partial jacobian $\frac{\partial (\varphi_1,\ldots
\varphi_n)}{\partial (\lambda_1,\ldots, \lambda_n)}$ does not vanish.
Therefore, at least locally, we can invert with respect to the last
variables,
$$\lambda = \psi(x^{(1)}(t),\ldots,x^{(r)}(t),x(t)).$$
From that, the components $\psi_i$ of $\psi$ are first integrals of
the lifted vector field $\vec X^{r+1}$. We consider the sheaf
of rings $\Psi$ generated by these functions $\psi_i$. This is a sheaf of
regular rings of dimension $n$. By construction $\mathcal F_{\Psi}$
is transversal to the fibers of the projection $M^r\times M \to M^r$.
We conclude that this sheaf is a local superposition law for $\vec
X$.
\end{proof}

\section{Local Superposition Theorem}

\begin{theorem}[(Local Superposition Theorem \cite{CGM2007})]
\label{C1THElocallie}
The  non-autonomous vector field $\vec X$ in $M$ admits a 
local superposition law if and only its enveloping
algebra is finite dimensional.
\end{theorem}

  Once we understand the relation between superposition law and local 
superposition law, we see that Lie's original proof is still valid. 
Here we follow \cite{Lie1893}. However we detail some points that were not
explicitly detailed  and remain obscure in the original proof.

  It is clear that the finiteness of the enveloping algebra is not
a sufficient condition for the existence of a superposition law. It
is neccesary to integrate the enveloping algebra into a Lie group action.

For instance, any non-linear \emph{autonomous} vector field has a one dimensional
enveloping Lie algebra, but it is widely know that the knowldedge of
a number of integral curves of a generic non-linear vector field
do not lead to the general solution. Some people can argue that
this case is too degenerated because the enveloping algebra 
of an autonomous vector field is too small in dimension. 
However, we can construct of a differential equation in
dimension 4, having an enveloping Lie algebra of dimension 4
for which no superposition Law is expectable. 

\begin{example}\label{exampleLorentz} Let us consider the direct product
of the well known Lorentz system and a Riccati equation
\begin{eqnarray*}\label{LR}
\dot x &=& \sigma(y-x)\\ 
\dot y &=& x(\rho-z) - y \\
\dot z &=& xy -\beta z \\
\dot u &=& \frac{2u}{t} - \frac{2}{t^2} - u^2
\end{eqnarray*}
being $\sigma,\, \beta,$ and $\rho$ parameters. We can apply Lie test
and we arrive to de conclusion that the enveloping algebra is of 
dimension 4 and have the following system of generators. 
\begin{eqnarray*}
\vec{X_1} &=&  \sigma(y-x)\frac{\partial}{\partial_x} + (x(\rho-z)-y)\frac{\partial}{\partial y} + (xy-\beta z)\frac{\partial}{\partial z} \\
\vec{X_2} &=& u^2 \frac{\partial}{\partial u}\\
\vec{X_3} &=& u\frac{\partial}{\partial u}\\
\vec{X_4} &=& \frac{\partial}{\partial u}
\end{eqnarray*}
Our original vector field is $\vec X = \vec X_1 - \vec X_2 + \frac{2}{t}\vec X_3 - \frac{2}{t^2}\vec X_4$. Let us
assume that this system admits a superposition Law,
$$\varphi\colon (\mathbb C^4)^r\times \mathbb C^4 \to \mathbb C^4.$$
The solutions of the Riccati equation in (\ref{LR}) can be easily found,
and they turn to be of the form $u(t) = \frac{1 + 2at}{t(1+at)}$ with $a$ constant. Therefore,
the solutions of (\ref{LR}) are of the form $\left(x(t),y(t),z(t),\frac{1 + 2at}{t(1+at)}\right)$
where $(x(t),y(t),z(t))$ is a solution of the Lorentz system. By taking
suitable values for $a_1,\ldots,a_r$, we can put the corresponding solutions 
of the Riccati equation inside the superposition law, project $\mathbb C^4$ onto
$\mathbb C^3$, and we will get a map $\bar\varphi$ of the same kind that $\varphi$,
$$\bar \varphi\colon (\mathbb C^3)^r \times \mathbb C^4 \to \mathbb C^3, \quad (x_i,y_i,z_i,\lambda) 
\mapsto \pi_{x,y,z}\left(\varphi \left(x_i, y_i, z_i, \frac{1+2a_it}{t(1+a_it)}\right)\right)$$
that gives us the general solution of the Lorentz system in function of
$r$ known solutions and $4$ arbitrary constants. Of course, the existence
of such a map for the Lorenz atractor equations is not expectable at all. 
\end{example}

\subsubsection*{\begin{center} Local Superposition Law Implies Finite Dimensional \\
Enveloping Algebra \end{center}}\label{ATLocalSuperpositionimpliesLVG}

First, let us assume that $\vec X$ admits a local superposition law
$\Psi$ in $\mathcal O_{M^{r+1}}$. Let us consider the cartesian
power $\vec X^{r+1}$ as a non-autonomous vector field in $M^{r+1}$.
Any section $\psi$ of $\Psi$ is a first integral of $\vec X^{r+1}$.

 Consider the family of vector fields
$\{\vec X^{r+1}_{t}\}_{t\in S}\subset \mathfrak X(M^{r+1})$. Let us
take a maximal subfamily  $\{\vec X^{r+1}_{t}\}_{t\in \Lambda}$ of
$\mathcal O_{M^{r+1}}$-linearly independent vector fields; here,
$\Lambda$ is some subset of $S$. The cardinal of a set $\mathcal
O_{M^{r+1}}$-linearly independent vector fields is bounded by the
dimension of $M^{r+1}$. Hence,
  $$\Lambda = \{t_1,\ldots,t_m\},$$
is a finite subset. We denote the corresponding vector fields as
follows:
  $$\vec X^{r+1}_i = \vec X^{r+1}_{t_i} \quad i=1,\ldots,m.$$

 We consider the following notation; in the cartesian power $M^{r+1}$ we
denote the different copies of $M$ as follows:
  $$M^{r+1} = M^{(1)} \times \ldots \times M^{(r)} \times M.$$
We recall that the \emph{lifted} vector field $\vec X^{r+1}_i$ is
the sum,
  $$\vec X_i^{r+1} = \vec X_i^{(1)} + \ldots + \vec X_i^{(r)} + \vec X_i,$$
of $r+1$ copies of the same vector field acting in the different
copies of $M$.

 The sheaf $\Psi$ consists of first integrals of the fields $\vec X^{r+1}_i$.
Thus, for all $i,j$, the Lie bracket $[\vec X_i^{r+1},\vec
X_j^{r+1}]$ also annihilates the sheaf $\Psi$. We can write the Lie
bracket as a sum of $r+1$ copies of the same vector field in $M$.
Let us note that the Lie bracket of two lifted vector fields is
again a lifted vector field.  
$$[\vec X_i^{r+1},\vec X_j^{r+1}] = [\vec X_i^{(1)},\vec X_j^{(1)}]
+ \ldots + [\vec X_i^{(r)},\vec X_j^{(r)}] + [\vec X_i,\vec X_j].$$

We recall that a regular distribution of vector fields is a distribution 
of constant rank, and a regular distribution is said Frobenius integrable if
it is closed by Lie brackets. 

The $m$ vector fields $\vec X^{r+1}_i$ span a distribution which is 
generically regular of rank $m$.
However, in the general case, it is not Frobenius integrable. We add
all the feasible Lie brackets, obtaining an infinite family
$$ \mathcal X = \bigcup_{k=0}^\infty \mathcal X_k$$
where, $$\mathcal X_0 = \{\vec X^{r+1}_1,\ldots,\vec X^{r+1}_m\},$$
and, $$\mathcal X_{k} = \{[\vec Y, \vec Z]\,\, | \,\,\vec Y \in
\mathcal X_i, \vec Z \in \mathcal X_j \,\,\mbox{for}\,\,  i < k
\,\,\mbox{and}\,\,j < k\}.$$
 The family $\mathcal X$ is a set of \emph{lifted} vector fields.
They span a generically regular distribution in $M^{r+1}$ which is,
by construction, Frobenius integrable. We extract of this family a
maximal subset of $\mathcal O_{M^{r+1}}$-linearly independent vector
fields,
  $$\{\vec Y^{r+1}_1,\ldots, \vec Y^{r+1}_s\}.$$
They span the same generically regular Frobenius integrable
distribution of rank $s\geq m$. Without any lose of generality we
can assume that the $m$ vector fields $\vec X^{r+1}_i$ are within
this family. These vector fields $\vec Y^{r+1}_i$ annihilate the
sheaf $\Psi$. Thus, we also obtain \emph{the Lie inequality},
\begin{equation}\label{ATeq5}
s \leq nr
\end{equation}
because the distribution of vector fields annihilating $\Psi$ has
rank $nr + n - \dim(\Psi)$, and by hypothesis $\dim(\Psi)$ is
greater than $n$.

The vector fields $\vec Y_i^{r+1}$ in $M^{r+1}$ are \emph{lifted}.
We write them as a sum of different copies of a vector field in $M$.
\begin{equation}\label{ATeq6}
\vec Y^{r+1}_i = \vec Y_i^{(1)} + \ldots + \vec Y_i^{(r)} + \vec
Y_i.
\end{equation}
They span an integrable distribution, so that there exist analytic
functions $\lambda_{ij}^k$ in $M^{r+1}$ such that,
\begin{equation}\label{ATeq7}
[\vec Y^{r+1}_i, \vec Y^{r+1}_j ] = \sum_{k=1}^s
\lambda_{ij}^k(x^{(1)},\ldots,x^{(r)},x) \vec Y^{r+1}_k,
\end{equation}

Let us prove that these functions $\lambda_{ij}^k$ are, in fact,
constants. From \eqref{ATeq6} we have
$$[\vec Y^{r+1}_i, \vec Y^{r+1}_j ] = [\vec Y^{(1)}_i, \vec Y^{(1)}_j] + \ldots
+ [\vec Y^{(r)}_i, \vec Y^{(r)}_j]+ [\vec Y_i, \vec Y_j]$$ and then
substituting again \eqref{ATeq6} and \eqref{ATeq7} we obtain for all
$a = 1,\ldots, r$
\begin{equation}\label{ATeq8}
[\vec Y^{(a)}_i, \vec Y^{(a)}_j] = \sum_{k=1}^s
\lambda_{ij}^k(x^{(1)},\ldots,x^{(r)},x) \vec Y^{(a)}_k
\end{equation}
and
\begin{equation}\label{ATeq9}
 [ \vec Y_i , \vec Y_j]  = \sum_{k=1}^s
\lambda_{ij}^k(x^{(1)},\ldots,x^{(r)},x) \vec Y_k
\end{equation}
These vector fields act exclusively in their respective copies of
$M$. Then $\sum_{k=1}^s \lambda_{ij}^k \vec Y_k$ is a vector field in
$M$, and $\sum_{k=1}^s \lambda_{ij}^k \vec Y^{(a)}_k$ is  also a
vector field in $M^{(a)}$, the $a$-th component in the cartesian
power $M^{r+1}$. There is an expression in local coordinates for the
vector fields $\vec Y_i$,
$$\vec Y_i = \sum_{l=1}^n \xi_{il}(x)\frac{\partial}{\partial x_l},$$
and
$$\vec Y_i^{(a)} = \sum_{l=1}^n \xi_{il}(x^{(a)})\frac{\partial}{\partial x_l}.$$
Thus,
$$\sum_{k=1}^s \lambda_{ij}^k(x^{(1)},\ldots,x^{(r)},x) \vec Y_k = \sum_{k=1}^s\sum_{l=1}^n
\lambda_{ij}^k(x^{(1)},\ldots,x^{(r)},x)\xi_{kl}(x)\frac{\partial}{\partial
x_l}$$ is a vector field in $M$, and for each $a$,
$$\sum_{k=1}^s \lambda_{ij}^k(x^{(1)},\ldots,x^{(r)},x) \vec Y^{(a)}_k = \sum_{k=1}^s\sum_{l=1}^n
\lambda_{ij}^k(x^{(1)},\ldots,x^{(r)},x)\xi_{kl}(x^{(a)})\frac{\partial}{\partial
x_l^{(a)}}$$ is a vector field in $M^{(a)}$. Therefore, the
coefficients of $[\vec Y_i,\vec Y_j]$,
\begin{equation}\label{ATeq10}
\sum_{k=1}^s \lambda_{ij}^k(x^{(1)},\ldots,x^{(r)},x) \xi_{kl}(x);
\end{equation}
depend only on $x_1,\ldots,x_n$; and analogously for each $a$
varying from $1$ to $r$ the coefficients of $[\vec Y^{(a)}_i,\vec
Y^{(a)}_j]$,
\begin{equation}\label{ATeq11}
\sum_{k=1}^s
\lambda_{ij}^k(x^{(1)},\ldots,x^{(r)},x)\xi_{kl}(x^{(a)})
\end{equation}
depend only on the coordinate functions
$x_1^{(a)},\ldots,x_n^{(a)}$. Fix $1\leq \alpha\leq n$. Let us prove
that $\frac{\partial \lambda_{ij}^k}{\partial x_\alpha} = 0$ for all
$i=1,\ldots, s$, $j=1,\ldots, s$ and $k=1,\ldots, s$. The same
argument is valid for $\frac{\partial \lambda_{ij}^k}{\partial
x^\beta_\alpha}$, just interchanging the factors of the cartesian
power $M^{r+1}$.

The expressions \eqref{ATeq11} do not depend on $x_{\alpha}$, and
then their partial derivative vanish,
\begin{equation}\label{ATeq12}
\sum_{k=1}^s \frac{\partial \lambda_{ij}^k}{\partial
x_{\alpha}}\xi_{kl}(x^{(a)}) = 0.
\end{equation}
We can write these expressions together in matrix form,
\begin{equation}\label{ATeq13}
\left(\begin{matrix}
\xi_{11}(x^{(1)}) & \xi_{21}(x^{(1)}) & \ldots & \xi_{s1}(x^{(1)})\\
\xi_{12}(x^{(1)}) & \xi_{22}(x^{(1)}) & \ldots & \xi_{s2}(x^{(1)})\\
\vdots & \vdots & \ddots & \vdots \\
\xi_{1n}(x^{(1)}) &\xi_{2n}(x^{(1)}) & \ldots & \xi_{sn}(x^{(1)})\\
\xi_{11}(x^{(2)}) & \xi_{21}(x^{(2)}) & \ldots & \xi_{s1}(x^{(2)})\\
\vdots & \vdots & & \vdots \\
\vdots & \vdots & & \vdots \\
\xi_{1n}(x^{(r)}) &\xi_{2n}(x^{(r)}) & \ldots & \xi_{sn}(x^{(r)})
\end{matrix}\right)
\left(\begin{matrix} \frac{\partial \lambda^1_{ij}}{\partial
x_{\alpha}} \\ \vdots 
\\ \vdots \\ \vdots \\ \vdots \\ \frac{\partial \lambda^s_{ij}}{\partial
x_{\alpha}}
\end{matrix}\right) = \left(\begin{matrix} 0 \\ \vdots \\ \vdots  \\ \vdots \\ \vdots \\ \vdots \\ 0 \end{matrix}
\right);
\end{equation}
for all $i = 1,\ldots, s$ and $j = 1,\ldots, s$. Then, the vector
$\left(\frac{\partial \lambda_{ij}^1}{\partial x_\alpha},\ldots,
\frac{\partial \lambda_{ij}^s}{\partial x_{\alpha}}\right)$ is a
solution of a linear system of $n\cdot r$ equations. If we prove
the matrix of coefficients above is of maximal rank, then from Lie's
inequality \eqref{ATeq5} we know that this system has only a trivial
solution and $\frac{\partial \lambda_{ij}^k}{\partial
x_{\alpha}}=0$.

\medskip
In order to do that, let us consider the natural projections,
$$\pi\colon  M^{(1)}\times \ldots \times M^{(r)} \times M
\to M^{(1)} \times \ldots \times M^{(r)},$$ onto the first $r$
factors and
$$\pi_1\colon M^{(1)}\times \ldots \times M^{(r)} \times M \to M,$$
onto the last factor.

Vector fields $\vec Y_k^{r+1}$ are projectable by $\pi$, and their
projection is
$$\pi_*(\vec Y_k^{r+1}) = \vec Y_k^{(1)} + \ldots +\vec Y_k^{(r)},$$
The expression of $\pi_*(\vec Y_k^{r+1})$ in local coordinates is
precisely,
$$\pi_*(\vec Y^{r+1}_k) = \sum_{l=1}^n \sum_{m=1}^r \xi_{kl}(x^{(m)}) \frac{\partial}{\partial
x^{(m)}_l}$$ the $k$-th column of the matrix of coefficients of
\eqref{ATeq13}. We can state that $\frac{\partial \lambda_{ij}^k}{\partial x_\alpha}$ 
vanish if and only if the vector fields $\pi_*(\vec Y_1^{r+1}), \ldots, \pi_*(\vec
Y_s^{r+1})$ are generically linearly independents. Assume that there is a
non-trivial linear combination equal to zero,
$$\sum_{k=1}^{s} F_k(x^{(1)},\ldots x^{(r)})\pi_*(\vec Y_k^{r+1}) =
\sum_{k=1}^s\sum_{m=1}^r F_k(x^{(1)},\ldots x^{(r)})Y_k^{(m)} = 0$$
the fields $\vec Y^{r+1}_k$ are generically linearly independent,
and then the vector field
$$\vec Z = \sum_{k=1}^s F_k(x^{(1)},\ldots x^{(r)})\vec Y_k^{r+1} = \sum_{k=1}^s F_k(x^{(1)},\ldots x^{(r)})\vec
Y_k$$ is different from zero. The vector field $\vec Z$ annihilates
simultaneously the sheaves $\Psi$ and $\pi_1^*(\mathcal O_M)$. It
annihilates $\Psi$ because it is a linear combination of the $\vec
Y_{k}^{r+1}$. Moreover it annihilates $\pi_1^{*}(\mathcal O_M)$, because
it is a linear combination of the $\vec Y_k$. This is in contradiction
with the transversality condition of Definition \ref{ATdef6}. 

We have proven that the $\pi_*(\vec Y^{r+1}_n), \ldots \pi_*(\vec
Y^{r+1}_n)$ are linearly independent, and then that the partial
derivatives $\frac{\partial \lambda_{ij}^k}{\partial x_\alpha}$
vanish. If we fix a superindex $(a)$ between $1$ and $r$, the same
argument is valid for the partial derivatives $\frac{\partial
\lambda_{ij}^k}{\partial x^{(a)}_{\alpha}}$. Hence, the
$\lambda_{ij}^k$ are constants $c_{ij}^k \in \mathbb C$,
\begin{equation}
[\vec Y_i, \vec Y_j] = c_{ij}^k\vec Y_k,
\end{equation}
and the vector fields $\vec Y_1$,$\ldots$,$\vec Y_s$ span a
$s$-dimensional Lie algebra in $M$, the enveloping algebra
of $\vec X$.

\subsubsection*{Finite Dimensional Enveloping Algebra Implies Local Superposition}

Reciprocally, let us assume that $\vec X$ has a finite dimensional
enveloping algebra. Throughout this section we assume that
this enveloping algebra acts \emph{transitively} in $M$. It
means that for a generic point $x$ of $M$ the values of elements of
$\mathfrak g(\vec X)$ at $x$ span the whole tangent space:
$$\mathfrak g(\vec X)_x = T_x M.$$

This assumption of transitivity can be done in the local case
without any lose of generality. If the $\mathfrak g(\vec X)$ is not
transitive, it gives a non-trivial foliation of $M$. We can then
substitute the integral leaves of the foliation for the phase space
$M$.

\begin{lemma}
  Let us consider $\vec Y_1, \ldots, \vec Y_m$ vector fields in $M$,
$\mathbb C$-linearly independent but generating a distribution which
is generically of rank less than $m$. Then there is a natural number
$r$ such that the lifted vector fields $\vec Y_1^{r},\quad, \vec
Y_m^{r}$ generate a generically regular distribution of rank $m$ of
vector fields of $M^r$.
\end{lemma}

\begin{proof}
We use an induction argument on $m$. The initial case is true,
because a non null vector field generate a distribution which is
generically regular of rank 1. Let us assume that the lemma is
proven for the case of $m$ vector fields. Consider $m+1$ vector
fields  $\vec Y_1,\ldots, \vec Y_m, \vec Y$. We can substitute some
cartesian power $M^r$ for $M$ and apply the induction hypothesis on
$\vec Y_1,\ldots, \vec Y_m$. Thus, we can assume that $Y_1^r,\ldots, Y_m^r$
span a generically regular distribution of vector fields in $M^ r$.
From now on, we denote $N$ for the cartesian power $M^r$, and $\vec Z_i$
for the lifted vector field $\vec Y_i^r$.

If the distribution spanned by $\vec Z_1,\ldots, \vec Z_m, \vec Z$
is not of rank $m+1$, then $\vec Z$ can be written as a linear combination
of the others with coefficients functions in $N$,
$$\vec Z = \sum_{i=1}^m f_i(x)\vec Z_i.$$
Consider $N^2 = N^{(1)} \times N^{(2)}$. Let us prove that the
distribution generated by $\vec Z^2_1,\ldots, \vec Z^2_m, \vec Z^2$
is generically regular of rank $m+1$. Using \emph{reductio ad
absurdum} let us assume that the rank is less that $m+1$. Then $Z^2$
is a linear combination of $\vec Z^2_1,\ldots, \vec Z^2_m$ with
coefficients functions in $N^2$,
\begin{equation}\label{ATeq14}
\vec Z^2 = \sum_{i=1}^m g_i(x^{(1)},x^{(2)})(\vec Z^{(1)}_i + Z^{(2)}_i),
\end{equation}
 on the other hand,
\begin{equation}\label{ATeq15}
\vec Z^2 = \vec Z^{(1)} + \vec Z^{(2)} = \sum_{i=1}^m f_i(x^{(1)})
\vec Z^{(1)}_i + f_i(x^{(2)}) \vec Y^{(2)}_i.
\end{equation}
Equating \eqref{ATeq14} and \eqref{ATeq15} we obtain,
$$g_i(x^{(1)},x^{(2)}) = f_i(x^{(1)}) = f_i(x^{(2)}), \quad i=1,\ldots,m.$$
Hence, the functions $g_i$ are constants $c_i\in \mathbb C$ but in
such case,
$$\vec Z = \sum_{k=1}^m c_i\vec Z_i,$$
$\vec Z$ is $\mathbb C$-linear combination of $\vec Z_1,\ldots \vec
Z_m$, in contradiction with the hypothesis of the lemma of $\mathbb
C$-linear independence of these vector fields. Then, the rank of the
distribution spanned by $\vec Y_1^{2r}$,$\ldots$,$\vec Y^{2r}_m$,$\vec
Y^{2r}$ is $m+1$.
\end{proof}

Let us take a basis $\{\vec Y_1,\ldots \vec Y_s\}$ of $\mathfrak
g(\vec X)$. The previous lemma says that there exist $r$ such that
the distribution generated by $\vec Y^r_1,\ldots,\vec Y^r_s$ is
generically regular of rank $s$ in $M^r$, the dimension of $M^r$ is
$nr$ so that we have again $s\leq nr$. We take one additional factor
in the cartesian power, and hence we consider the complex analytic
manifold $M^{r+1}$. We define $\Psi$ as the sheaf of first integrals
of the lifted vector fields $\vec Y^{r+1}_1, \ldots, Y^{r+1}_s$ in
$M^{r+1}$. These vector fields span a generically Frobenius
integrable distribution of rank $s$, and then $\Psi$ is a sheaf of
regular rings of dimension $n + nr - s$, which is greater or equal to 
$n$.

Let us see that the foliation $\mathcal F_{\Psi}$ is generically
transversal to the fibers of the projection $\pi\colon M^{r+1}\to
M^r$. By the \emph{transitivity} hypothesis on
$\mathfrak g(\vec X)$, the tangent space to these fibers is spanned
by the $s$ vector fields $\vec Y_k$, where 
$k$ varies from $1$ to $s$. Let $\vec Z$ be a
vector field tangent to the fibers of $\pi$. If this vector field
is tangent to $\mathcal F_{\Psi}$ then it is a linear combination of
the $\vec Y^{r+1}_i$ with coefficients in $\mathcal O_{M^{r+1}}$,
$$\vec Z = \sum_{i=1}^s F_i(x^{(1)},\ldots,x^{(r)},x) \vec Y_i^{r+1} =
\sum_{i=1}^{s} F_i Y_{i}^{r} + \sum_{i=1}^s F_iY_i,$$
and from that we obtain that,
\begin{eqnarray}
\sum_{i=1}^s F_i(x^{(1)},\ldots,x^{(r)},x)\vec  Y^r_i &=& 0, \label{ATeq17}\\
\sum_{i=1}^s F_i(x^{(1)},\ldots,x^{(r)},x)\vec Y_i^{r+1} &=&
\sum_{i=1}^s F_i(x^{(1)},\ldots,x^{(r)},x)\vec Y_i \label{ATeq18}
\end{eqnarray}
where $\vec Y^r_i =\vec Y^{(1)}_i + \ldots + \vec Y^{(r)}_i$. The
vector fields $\vec Y_1^r$,$\ldots$,$\vec Y^r_s$ span a distribution
which is generically of rank $s$ by construction. And, if considered
as vector fields in $M^{r+1}$, they do not depend of the last factor
in the cartesian power. Then we can specialize the functions $F_i$
to some fixed value $x\in M$ in the last component, for which the
linear combination \eqref{ATeq17} is not trivial. We obtain an
expression,
\begin{equation}\label{ATeq19}
\sum_{i=1}^s G_i(x^{(1)},\ldots,x^{(r)})\vec  Y_i = 0,
\end{equation}
that gives us is a linear combination of the $\vec Y_i$ with
coefficients functions if $M^r$. Vector fields $\vec Y_i$ are
$\mathbb C$-linearly independent; therefore this linear combination
\eqref{ATeq19} is trivial and the functions $G_i$ vanish. These
functions are the restriction of the functions $F_i$ to arbitrary
values in the last factor of $M^{r+1}$, then the functions $F_i$
also vanish. We conclude that the vector field $\vec Z$ is zero, and
$\mathcal F_{\Psi}$ intersect transversally the fibers of the projection
$\pi$.

\section{Actions of complex analytic Lie groups}

\subsection{Complex Analytic Lie Groups}

\begin{Definition}
 A complex analytic Lie Group $G$ is a group endowed with an structure
of complex analytic manifold. This structure is compatible with the
group law in the sense of that the map:
$$G\times G \to G, \quad (\sigma,\tau)\mapsto \sigma\tau^{-1}$$
is complex analytic.
\end{Definition}

  From now on, let us consider a complex analytic Lie group $G$.
 We say that a subset $H$ is a complex analytic subgroup of $G$ if
it is both a subgroup and a complex analytic submanifold of $G$. In
the usual topology, the closure $\bar H$ of a subgroup of $G$ is
just a Lie subgroup, and it is not in general complex analytic.
Instead of $\bar H$ we consider the \emph{complex analytic
closure},
$$\tilde H = \bigcap_{H\subset G'} G'\qquad  G' \mbox{ complex analytic.}$$
This is the smaller complex analytic subgroup containing $H$.

\subsubsection*{Invariant Vector Fields}

  Each element $\sigma \in G$ defines two biholomorphisms of $G$ as
  complex analytic manifold. They are the right translation $R_{\sigma}$ and
  the left translation $L_{\sigma}$:
  $$R_{\sigma}(\tau) = \tau\sigma, \quad L_{\sigma}(\tau) =
  \sigma\tau.$$

  The following result is true in the general theory of groups. It
is quite interesting that right translations are the symmetries of
left translations and viceversa. This simple fact has amazing
consequences in the theory of differential equations.

\begin{remark}
  Let $f\colon G\to G$ be a map. Then, there is an element $\sigma$
  such that $f$ is the right translation $R_\sigma$ if and only if
  $f$ commutes with all left translations.
\end{remark}

  Let us consider $\mathfrak X(G)$ the space of all regular
complex analytic vector fields in $G$. It is, in general, an
infinite dimensional Lie algebra over $\mathbb C$ -- with the Lie
bracket of derivations --. A biholomorphism $f$ of $G$ transform
regular vector field into regular vector fields. By abuse on notation
we denote by the same symbol $f$ the automorphism of the Lie algebra
of regular vector fields.
$$f\colon \mathfrak X(G) \to \mathfrak X(G),\qquad
\vec X \mapsto f(\vec X) \quad f(\vec X)_\sigma = f'(\vec
X_{f^{-1}(\sigma)})$$

\begin{Definition}
  A regular vector field $\vec A\in\mathfrak X(G)$ is called a right
invariant vector field if for all $\sigma\in G$ the right
translation $R_{\sigma}$ transform $\vec A$ into itself. A regular
vector field $\vec A\in \mathfrak X(G)$ is called a left invariant
vector field if for all $\sigma\in G$ the right translation
$L_{\sigma}$ transforms $\vec A$ into itself.
\end{Definition}

  For any tangent vector $\vec A_e\in T_eG$ in the identity element
of $G$ there is an only right invariant vector field $\vec A$ which
takes the value $\vec A_e$ at $e$. Then the space of right
invariants vector fields is a $\mathbb C$-vector space of the same
dimension as $G$. The Lie bracket of right invariant vector fields
is a right invariant vector field. \emph{We denote by $\mathcal
R(G)$ the Lie algebra of right invariants vector fields in $G$.} The
same discussion holds for left invariants vector fields. \emph{We
denote by $\mathcal L(G)$ the Lie algebra of left invariants vector
fields in $G$.}

\medskip

  Note that the inversion morphism that sends $\sigma$ to
$\sigma^{-1}$ is an involution that conjugates right translations
with left translations. Then, it sends right invariant vector fields
to left invariant vector fields. Thus, $\mathcal R(G)$ and $\mathcal L(G)$ 
are isomorphic Lie algebras.

\subsubsection*{Exponential Map}

  Let us consider $\vec A\in \mathcal R(G)$. There is a germ of
solution curve of $\vec A$ that sends the origin to the identity
element $e\in G$. By the composition law, it extend to a globally
defined curve:
$$\sigma\colon \mathbb C \to G, \quad t\mapsto \sigma_{t}.$$
This map is a group morphism and $\{\sigma_{t}\}_{t\in\mathbb C}$ is
a monopoarametric subgroup of $G$.

\begin{Definition}
  We call exponential map to the application that assigns to each
  right invariant vector field $\vec A$ the element $\sigma_1$ of
  the solution curve $\sigma_t$ of $\vec A$:
  $$\exp\colon \mathcal R(G)\to G, \quad \vec A\mapsto \exp(\vec
  A)=\sigma_1.$$
\end{Definition}

  The exponential map is defined in the same way for left invariant
vector fields. If $\vec A\in \mathcal R(G)$ and $\vec B\in \mathcal
L(G)$ take the same value at $e$, then they share the same solution
curve $\{\sigma_{t}\}_{t\in\mathbb C}$, and $\exp(t\vec A) =
\exp(t\vec B)$.

\medskip

  If $\vec A$ is a right invariant vector field, then the exponential of
$\vec A$ codifies the flow of $\vec A$ as vector field. This flow is
given by the formula:
$$\Phi^{\vec A}\colon \mathbb C\times G \to G,\quad \Phi^{\vec A}\colon (t,\sigma)\mapsto
\exp(t\vec A)\cdot \sigma.$$
  The same is valid for left invariant vector fields. If $\vec B$ is
a left invariant vector field then the flow of $\vec B$ is:
$$\Phi^{\vec B}\colon \mathbb C\times G \to G, \quad \Phi^{\vec
B}\colon (t,\sigma)\mapsto \sigma\cdot \exp(t\vec B).$$ 
Finally we see that \emph{right invariant vector fields are infinitesimal
generators of monoparametric groups of left translations}.
Reciprocally left invariant vector fields are infinitesimal
generators of monoparametric groups of right translations. The following
statement is the infinitesimal version of the commutation of right and left
translations:

\begin{remark}\label{THLeftRight}
  Let $\vec X$ be a vector field in $G$. It is a right invariant vector field if and only if for
all left invariant vector field $\vec B$ in $G$ the Lie bracket,
$[\vec B, \vec X]$, vanish. 
\end{remark}

The same is true if we change the roles of right and left  in the
previous statement. We have seen that the Lie algebra of symmetries
of right invariant vector fields is the algebra of left invariant
vector fields and viceversa:
$$[\mathcal R(G),\mathcal L(G)] = 0.$$

\subsection{Complex Analytic Homogeneous Spaces}

\begin{Definition}
A $G$-space $M$ is a complex analytic manifold $M$ endowed with a
complex analytic action of $G$,
$$G\times M \xrightarrow{a} M, \quad (\sigma,x)\mapsto \sigma\cdot
x.$$
\end{Definition}

  Let $M$ be a $G$-space. For a point $x\in M$ we denote by $H_x$
the \emph{isotropy subgroup} of elements of $G$ that let $x$ fixed, 
and $O_x$ for the \emph{orbit} $G\cdot x$ of $x$. Let us remind that an action
is said to be \emph{transitive} if there is a unique orbit; it is said \emph{free} if
for any point $x$ the isotropy $H_x$ consist of the identity element $e\in G$. The kernel
$K$ of the action is the intersection of all isotropy subgroups $\cap_{x\in M}H_x$. It is
a normal subgroup of $G$. The action of $G$ in $M$ is said to be \emph{faithful} 
if its kernel consist only of the identity element.    

\begin{Definition}
A $G$-space $M$ is a homogeneous $G$-space if the action of $G$ in
$M$ is transitive. A principal homogeneous $G$-space $M$ is a 
$G$-space such that the action of $G$ in $M$ is free and transitive.
\end{Definition}

\medskip 

By a morphism of $G$-spaces we mean a morphism $f$ of complex
analytic manifolds which commutes with the action of $G$,
$f(\sigma\cdot x) = \sigma\dot f(x)$. Whenever it does not lead to 
confusion we write homogeneous space instead of
$G$-homogeneous space on so on.

\medskip
  Let $M$ be a $G$-space. For any $x$ in $M$ the natural map,
$$G \to O_x\quad  \sigma \mapsto \sigma\cdot x$$
induces an isomorphism of $O_x$ with the homogeneous space of cosets
$G/H_x$; it follows that any homogeneous space is isomorphic to 
a quotient of $G$.

\medskip

 A point $x\in M$ is called a \emph{principal point} if and
only if $H_x = \{e\}$ if and only if $O_x$ is a principal
homogeneous space. In such case we say that $O_x$ is a
\emph{principal orbit}.

\medskip

\begin{Definition}\label{ApABAnalyticQuotient}
  For a $G$-space $M$ we denote by $M/G$ the space 
orbits of the action of $G$ in $M$.
\end{Definition}

The main problem of the \emph{invariant theory} is the study of the
structure of such quotient spaces. In general this space of orbits
$M/G$ can have a very complicated structure, and then it should not
exist within the category of complex analytic manifolds. This
situation is even more complicated in the algebraic case.

\subsubsection*{Fundamental Vector Fields}\label{fundamental}

  Let $M$ be a $G$-space. An element $\vec A$ of the Lie algebra $\mathcal
  R(G)$ can be considered as a vector field in the cartesian product
  $G\times M$ that acts in the first component only. This vector
  field $\vec A$ is projectable by the action,
  $$a\colon G\times M \to M,$$
  and its projection in $M$ is a vector field that we denote $\vec
  A^M$, the \emph{fundamental vector field} in $M$ induced by $\vec
  A$. Namely, for any $x\in M$ we have:
  $$(\vec A^M)_x = \frac{d}{dt} (\exp({t\vec A})\cdot x).$$

\begin{proposition}
  The map that assign fundamental vector fields to right invariant vector fields, 
  $$\mathcal R(G) \to \mathfrak X(M), \quad \vec A \to \vec A^M,$$
  is a Lie algebra morphism.
\end{proposition}

  The kernel of this map is the Lie algebra of the kernel subgroup
  of the action, \emph{i.e.} $\cap_{x\in M}H_x$. Thus, if the action is faithful then it is
  injective; its image is isomorphic to the quotient $\mathcal
  R(G)/\mathcal R\left(\cap_{x\in M}H_x\right)$.

\begin{Definition}\label{ApBfundamentalvectorfields}
  We denote by $\mathcal R(G,M)$ the Lie algebra of
  fundamental vector fields of the action of $G$ on $M$.
\end{Definition}

  The flow of fundamental vector fields is given by the exponential map in
the group $G$:
$$\Phi^{\vec A^M}\colon \mathbb C \times M \to M, \quad (t,x)\mapsto
\exp(t\vec A)\cdot x.$$

\subsubsection*{Basis of an Homogeneous Space}

  Here we introduce the concept of \emph{basis} and \emph{rank} of
homogeneous spaces. This concept has been implicitely used by S. Lie and E. Vessiot in
their research on superposition laws for differential equations. However,
we have not found any modern reference on this point. From now on, let us consider 
an homogeneous space $M$.

\begin{Definition}
  Let $S$ be any subset of $M$. We denote by $H_S$ to the isotropy
  subgroup of $S$:
  $$H_S = \bigcap_{x\in S} H_x = \{\sigma\,|\,\sigma\cdot x = x\,\,\, \forall x\in S\}.$$
\end{Definition}

\begin{Definition}
  We call $\langle S \rangle$, space spanned by $S\subset M$, to the
  space of invariants of $H_S$:
  $$\langle S \rangle = \{x\, | \, \sigma\cdot x = x \,\,\, \forall \sigma\in H_S
  \}.$$
\end{Definition}

\begin{Definition}
  A subset $S\subset M$ is called a system of generators of $M$ as
  $G$-space if $\langle S \rangle = M$.
\end{Definition}

  $S$ is a system of generators if and only if the isotropy $H_S$ is
  coincides with the kernel of the action $H_M$,
  $$\bigcap_{x\in S} H_x = \bigcap_{x\in M} H_x.$$
  In particular, for a faithful action, 
$S$ is a system of generators of $M$ if an only if $H_S = \{e\}$.

\begin{Definition}\label{ApBbasis}
We call a \emph{basis} of $M$ to a minimal system of generators of
$M$. We say that an homogeneous space is of finite rank if there
exist a finite basis of $M$. The minimum cardinal of basis of $M$ is
called the rank of $M$.
\end{Definition}

  Let us consider an homogeneous space $M$. The group $G$ acts in
  each cartesian power $M^r$ component by component, so $M^r$ is a
  $G$-space. The following proposition characterizing basis of $M$ is
  elemental:

\begin{proposition}\label{pro5.16}
  Let us assume that $M$ is a faithful homogeneous space of finite
  rank $r$. Let us consider $\bar x = (x^{(1)},\ldots,x^{(r)}) \in M^r$. The
  following statements are equivalent:
\begin{enumerate}
\item[(1)] $\{x^{(1)},\ldots,x^{(r)}\}$ is a basis of $M$.
\item[(2)] $\bar x$ is a principal point of $M^r$.
\item[(3)] $H_{\bar x} = \{e\}$.
\end{enumerate}
\end{proposition}

\begin{corollary}[(Lie's inequality)]
  If $M$ is a faithful homogeneous space of finite rank $r$ then,
  $$r\cdot \dim M \geq \dim G.$$
\end{corollary}

\begin{proof}
  The dimension of $M^r$ is $r\cdot\dim M$; but $M^r$ contains principal
orbits $O_{\bar x}$ which are isomorphic to $G$, and then of the same dimension. 
\end{proof}

\begin{proposition}\label{ApBsetofbasis}
 Let $M$ be a faithful homogeneous space of finite rank $r$. The subset
$B\subset M^r$ of all pricipal points of $M^r$ is an analytic open subset.  
\end{proposition}

\begin{proof} 
Consider $\{\vec A_1,\ldots,
\vec A_s\}$ a basis of the Lie algebra $\mathcal R(G,M^r)$ of
fundamental fields in $M^r$. For each $\bar x\in M^r$, those vector
fields span the tangent space to $O_{\bar x}$,
$$\langle \vec A_{1,\bar x}, \ldots, \vec A_{s,\bar x}\rangle =
T_{\bar x}O_{\bar x}\subset T_{\bar x} M^r.$$ 
By Proposition \ref{pro5.16} there is a principal point $\bar y\in M^r$. The dimension of $O_{\bar y}$ is
$s$, so that $\vec A_{1,\bar y}$, $\ldots$, $\vec A_{s,\bar y}$ are
linearly independent. The set
$$B_0 = \{\bar x\in M^r\,|\, \vec A_{1,\bar x}, \ldots, \vec A_{s,\bar x}
\mbox{ are linearly independent }\},$$
is the complementary of the analytic closed subset defined by
the minors of rank $s$ of the matrix formed by the vectors $\vec A_i$. It
is non-void, because it contains $\bar y$. Thus, $B_0$ is a non-void
\emph{analytic open subset} of $M$.
 
  By definition, a $r$-frame $\bar x$ is in $B_0$ if and
only if its orbit $O_{\bar x}$ is of dimension $s$, if and
only if the isotropy subgroup of $\bar x$ is a discrete subgroup
of $G$; therefore $B\subset B_0 \subset M^r$.

\medskip

The function in $B_0$ that assigns to each $\bar x$ the cardinal $\#
H_{\bar x}$ of its isotropy subgroup is upper semi-continuous; thus,
it reach its minimum $1$, along an analytic open subset $B\subset
B_0$. This $B$ is the \emph{set of all principal points} of $M^r$,
it is an analytic open subset of $B_0$ and then also of $M^r$.
\end{proof}

\subsection{Pretransitive actions and Lie-Vessiot systems}

 From now on, let us consider a complex analytic Lie
group $G$, and a \emph{faithful} analytic action of $G$ on $M$,
$$G\times M \to M, \quad (\sigma, x) \to \sigma \cdot x.$$
$\mathcal R(G)$ is the Lie algebra of \emph{right-invariant} vector
fields in $G$. We denote $\mathcal R(G,M)$ the Lie algebra of
fundamental vector fields of the action of $G$ on $M$ 
(Definition \ref{ApBfundamentalvectorfields}).

For each $r\in\mathbb N$, $G$ acts in the cartesian power $M^r$
component by component. There is a minimum $r$ such that there are
principal orbits in $M^r$. If $M$ is an homogeneous space then this
number $r$ is the \emph{rank} of $M$ (Definition \ref{ApBbasis}).

\begin{Definition}\label{C1DEFpretransitive}
We say that the action of $G$ on $M$ is pretransitive if there
exists $r\geq 1$ and an analytic open subset $U\subset M^r$ such
that:
\begin{enumerate}
\item[(a)] $U$ is union of principal orbits.
\item[(b)] The space of orbits $U/G$ is a complex analytic manifold.
\end{enumerate}
\end{Definition}

\begin{proposition}\label{C1PROpretransitive}
If $M$ is a $G$-homogeneous space of finite rank, then the action of
$G$ in $M$ is pretransitive.
\end{proposition}

\begin{proof}
  Let $r$ be the rank of $M$. The set $B\subset M^r$ of \emph{principal points} 
of $M^r$ is an analytic open subset (Proposition
\ref{ApBsetofbasis}). Let us see that the space of orbits $B/G$
is a complex analytic manifold. In order to do so we will construct
an atlas for $B/G$. Let us consider the natural projection:
  $$\pi\colon B \to B/G, \quad \bar x\mapsto \pi(\bar x) = O_{\bar x}.$$
Let us consider $\bar x\in B$; let us construct a coordinate open
subset for $\pi(\bar x)$. $O_{\bar x}$ is a submanifold of $B$. We
take a submanifold $L$ of $B$ such that $O_{\bar x}$ and $L$
intersect transversally in $\bar x$,
$$T_{\bar x}U = T_{\bar x} O_{\bar x} \oplus T_{\bar x} L.$$

\emph{Let us prove that there is an polydisc $U_{\bar x}$, open in
$L$ and centered on $\bar x$, such that $O_{\bar x} \cap U_{\bar x}
= \bar x$.} The action of $G$ on $B$ is continuous and free. So
that, if the required statement statement holds, then there exists a
maybe smaller polydisc $V_{\bar x}$ centered in $\bar x$ and open in
$L$ such that for all $\bar y\in V_{\bar x}$ the intersection of
$O_{\bar y}$ with $V_{\bar x}$ is reduced to the only point $\bar
y$. Thus, $V_{\bar x}$ projects one-to-one by $\pi$.
$$\pi\colon V_{\bar x} \xrightarrow{\sim} V_{\pi(\bar x)} \qquad \bar
y \mapsto \pi(\bar y) = O_{\bar y}$$ $V_{\pi(\bar x)}$ is an open
neighborhood of $\pi(x)$. The inverse of $\pi$ is an homeomorphism
of such open subset of $B/G$ with a complex polydisc. In this way we
obtain an open covering of the space of orbits. Transitions
functions are the holonomy maps of the foliation whose leaves are
the orbits. Thus, transition functions are complex analytic and
$B/G$ is a complex analytic manifold.

\medskip

Reasoning by \emph{reductio ad absurdum} let us assume that there is
not a polydisc verifying the required hypothesis. Each neighborhood
of $x$ in $L$ intersects then $O_x$ in more than one points. Let us
take a topological basis,
$$U_1 \supset U_2 \supset \ldots \supset \ldots,$$
of open neighborhoods of $\bar x$ in $L$.
$$\bigcap_{i=1}^\infty U_i = \{\bar x\}, \quad U_i\cap O_{\bar x} = \{\bar x,
\bar x_i,\ldots \}.$$
 In this way we construct a sequence $\{\bar x_i\}_{i\in \mathbb N}$
in $L\cap O_{\bar x}$ such that,
$$\lim_{i\to\infty}\bar x_i = \bar x,$$
where $\bar x_i$ is different from $\bar x$ for all $i\in \mathbb
N$. The action of $G$ is free, so that, for each $i\in\mathbb N$
there is an unique $\sigma_i\in G$ such that $\sigma_i(\bar x_i) =
\bar x$. We write these $r$-frames in components,
$$\bar x = (x^{(1)},\ldots,x^{(r)}), \quad \bar x_i
=(x^{(1)}_i,\ldots,x^{(r)}_i),$$ we have that for all $i\in\mathbb
N$ and $k=1,\ldots, r,$
$$x^{(k)} = \sigma_i(x^{(k)}_i).$$
For all $i\in \mathbb N$ and $k=1,\ldots,r,$ let us denote, 
$$H_{x^{(k)}_i,x^{(k)}}=\{\sigma\in G\,|\,\sigma \cdot x^{(k)}_i = x^{(k)}\},$$
the set of all elements of $G$ sending $x^{(k)}_i$ to $x^{(k)}$.
This set is the image of the isotropy group $H_{x^{(k)}}$ by a right
translation in $G$. When $x^{(k)}_i$ is closer to $x^{(k)}$ this
translation is closer to the identity. We have,
$$\sigma_i \subset  \bigcap_{k=1}^r H_{x^{(k)}_i,x^{(k)}},$$
and from that,
$$\lim_{i\to\infty} \sigma_i \in \bigcap_{k=1}^r H_{x^{(k)}},$$
but this intersection in precisely the isotropy group of $\bar x$
and then $\sigma_i \xrightarrow{i\to\infty} e$.

On the other hand there is a neighborhood of $\bar x$ in $O_{\bar
x}$ that intersects $L$ only in $\bar x$. So that there is a
neighborhood $U_{e}$ of the identity in $G$ such that $U_{e}\cdot
\bar x$ intersects $L$ only in $\bar x$. From certain $i_0$ on,
$\sigma_i$ is in $U_{e}$, and then $\sigma_i = e$. Then $\bar
x_{i} = \bar x$, in contradiction with the hypothesis.
\end{proof}

\begin{remark}
  Pretransitive actions are not uncommon. For example, if $M$ and $N$ are
$G$ homogeneous spaces of finite rank, then $M\times N$ is not in general 
an homogeneous space, but it is clear that it is a pretransitive space. Also the
linear case provides useful examples. For a vector space $E$ it is clear that 
the action of the linear group $GL(E)$ on any tensor construction on $E$ is 
pretransitive.
\end{remark}

\begin{Definition}
  A non-autonomous analytic vector field $\vec X$ in $M$ is called a 
\emph{Lie-Vessiot system}, relative to the action of $G$, if its 
enveloping algebra is spanned by 
fundamental fields of the action of $G$ on $M$; 
\emph{i.e.} $\mathfrak g(\vec X)\subset \mathcal R(G,M)$.
\end{Definition}

\section{Global Superposition Theorem}\label{SectionGlobal}

In this section we prove our first main result. As mentioned,
it is inspired in Vessiot \cite{Vessiot1894} philosophy of
extracting the group action from the superposition law. 

\begin{theorem}\label{T1}
A non-autonomous complex analytic vector field $\vec X$
in a complex analytic manifold $M$ admits a superposition law if and only if
it is a Lie-Vessiot system related to a pretransitive Lie group action
in $M$.
\end{theorem}

\subsubsection*{Pretransitive Lie-Vessiot Systems admit Superposition Laws}

  Consider a faithful pretransitive Lie group action,
$$G\times M\to M,\quad (\sigma, x) \mapsto \sigma(x),$$
and assume that $\vec X$ is a Lie-Vessiot system relative to the
action of $G$. Let us consider the application,
$$\mathcal R(G) \to \mathfrak X(M), \quad \vec A \mapsto \vec A^M,$$
that send a right-invariant vector field in $G$ to its corresponding
fundamental vector field in $M$. Let $\{A_1,\ldots,A_s\}$ be a basis
of $\mathcal R(G)$, and denote by $\vec X_i$ their corresponding
fundamental vector fields $\vec A_i^M$. Then,
$$\vec X = \partial + \sum_{i=1}^s f_i(t)\vec X_i, \quad f_i(t)\in\mathcal
O_S.$$

\begin{Definition}\label{automorphic}
 The following vector field in $S\times G$,
$\vec A = \partial + \sum_{i=1}^s f_i(t)\vec A_i,$
is called the \emph{automorphic system} associated to $\vec X$.
\end{Definition}

From the definition of fundamental vector fields we know that if
$\sigma(t)$ is a solution of $\vec A$, then $\sigma(t)\cdot x_0$ is
a solution of $\vec X$ for any $x_0\in M$. 
The automorphic system is a particular case of a Lie-Vessiot system. 
The relation between solutions of the Lie-Vessiot system $\vec X$ and its 
associated automorphic system $\vec A$ leads to Lie's reduction method and representation 
formulas as exposed in \cite{Blazquez2008}.

\begin{proposition}\label{automorphicproperties}
The automorphic system satisfies:
\begin{itemize}
\item[(i)] Let $\vec A$ be the automorphic system associated to $\vec X$, and 
$\sigma(t)$ a solution of $\vec A$ defined in $S'\subset S$. Then, 
$\sigma(t)\cdot x_0$ is a solution of $\vec X$.
\item[(ii)] The automorphic system associated to  $\vec A$ is $\vec A$ itself.
\item[(iii)] The group law $G\times G\to G$ is the superposition law admitted by
the automorphic system.
\item[(iv)] If $\sigma(t)$ and $\tau(t)$ are two solutions of $\vec A$, then
$\sigma(t)^{-1}\tau(t)$ is a constant element of $G$. 
\end{itemize}
\end{proposition}

\begin{proof}
$(i)$ The tangent vector a $t=t_0$ to the curve $\sigma(t)\cdot x_0$ is precisely the projection
by the action of the tangent vector to $\sigma(t)$ at $t=t_0$. This is by definition,
$\vec A^{M}_{\sigma(t_0)\cdot x_0}$.

\noindent
$(ii)$ When we consider the case of a Lie group acting on itself, its algebra of fundamental
fields is its algebra of right invariant fields. We have $\vec A_i^G = \vec A_i$. 

\noindent
$(iii)$ It is a direct consequence of $(i)$.

\noindent
$(iv)$ Let $\sigma(t)$ and $\tau(t)$ be two solutions of $\vec A$. By $(iii)$ we know that the group
law is a superposition law. Thus, there is a constant $\lambda\in G$ such that $\sigma(t)\cdot \lambda = \tau(t)$.
Then, we have $\sigma(t)^{-1}\tau(t) = \lambda$.
\end{proof}

\medskip

There exists an analytic open subset $W\subset M^r$
which is union of principal orbits, and there exist the quotient $W/G$
as an analytic manifold. In such case the bundle,
$$\pi\colon W \to W/G,$$
is a \emph{principal bundle} modeled over the group $G$. Consider a
section $s$ of $\pi$ defined in some analytic open subset $V\subset
W/G$. Let $U$ be the preimage of $V$ for $\pi$. The section $s$
allow us to define a surjective function,
$$g\colon U \to G, \quad \bar x \to g(\bar x) \quad g(\bar x)\cdot s(\pi(\bar x)) = \bar x$$
which is nothing but the usual trivialization function. We define the following
maps $\varphi$ and $\psi$,
$$\varphi\colon U \times M \to M, \quad (\bar x ,y) \mapsto g(\bar
x)\cdot y$$
$$\psi\colon U \times M \to M, \quad (\bar x, y) \mapsto g(\bar
x)^{-1}\cdot y.$$

\begin{lemma}\label{C1LEM1.2.11}
  The section $\varphi$ is a superposition law for $\vec X$, and $\Psi =
\psi^*\mathcal O_M$ is a local superposition law for $\vec X$.
\end{lemma}

\begin{proof}
 Let us consider the map $g\colon U \to G$. By construction it is
a morphism of $G$-spaces: it verifies $g(\sigma\cdot\bar x) =
\sigma\cdot g(\bar x)$. Therefore, it maps fundamental vector fields
in $U$ to fundamental vector fields in $G$. Let us remember that
fundamental vector fields in $G$ are right invariant vector fields.
We have that $\vec X^r$ is projectable by $Id\times g$,
$$Id\times g \colon S \times U \to S \times G, \quad \vec X^{r}\mapsto \vec A,$$
and the image of $\vec X^r$ is the automorphic vector field $\vec
A$.

Finally, let us consider $r$ particular solutions of $\vec X$,
$x^{(1)}(t),\ldots,x^{(r)}(t)$. We denote by $\bar x(t)$ the curve
$(x^{(1)}(t),\ldots,x^{(r)}(t))$ in $M^{r}$. Therefore, $\bar x(t)$
is a solution of $\vec X^r$ in $M^r$. Hence, $g(\bar x(t))$ is a
solution of the automorphic system $\vec A$ in $G$, and for $x_0\in
M$ we obtain by composition $g(\bar x(t))\cdot x_0$, which is a
solution of $\vec X$. By the uniqueness of local solutions we know
that when $x_0$ varies in $M$ we obtain the general solution. Then,
$$\varphi(x^{(1)}(t),\ldots,x^{(r)}(t),x_0) = g(\bar x(t))\cdot
x_0$$ is the general solution and $\varphi$ is a superposition law
for $\vec X$.

With respect to $\psi$, let us note that it is a partial inverse for
$\varphi$ in the sense,
$$\psi(\bar x,\varphi(\bar x, x)) = x, \quad \varphi(\bar
x,\psi(\bar x, x))=x.$$ Hence, if $(\bar x(t),x(t))$ is a solution
of $\vec X^{r+1}$ then $\varphi(\bar x(t),\psi(\bar x(t),x(t)) =
x(t)$ is a solution of $\vec X$ and then $\psi(\bar x(t),x(t))$ is a
constant point $x_0$ of $M$, and $\Psi = \psi^*\mathcal O_M$ is a
sheaf of first integrals of $\vec X^r$.
\end{proof}

\subsubsection*{Superposition Law implies Pretransitive Lie-Vessiot}

Consider a superposition law for $\vec X$,
$$\varphi\colon U \times M \to M.$$
First, let us make some consideration on the open subset $U\subset
M^r$ in the definition domain of $\varphi$. Consider the family
$\{(\varphi_{\lambda},U_{\lambda})\}_{\lambda\in\Lambda}$ of
different superposition laws admitted by $\vec X$,
$\varphi_{\lambda}$ defined in $U_{\lambda}\times M$. There is a
natural partial order in this family: we write $\lambda < \gamma$ if
$U_\lambda\subset U_\gamma$ and $\varphi_{\lambda}$ is obtained from
$\varphi_{\gamma}$ by restriction:
$\varphi_{\gamma}|_{U_{\lambda}\times M} = \varphi_\lambda$. For a
totally ordered subset $\Gamma\subset \Lambda$ we construct a
supreme element by setting,
$$U_{\Gamma} = \bigcup_{\gamma\in \Gamma} U_{\gamma},$$
and defining $\varphi_{\Gamma}$ as the unique function defined in
$U_{\Gamma}$ compatible with the restrictions. By Zorn lemma, we can
assure that there exist a superposition law $\varphi$ defined in
some $U\times M$ which is maximal with respect to this order. From
now on \emph{we assume that the considered superposition law is
maximal.} It is clear that
$$\bar \varphi\colon U \times M^r \to M^r, \quad (\bar x, \bar y) \mapsto
(\varphi(\bar x, y^{(1)}), \varphi(\bar x, y^{(2)}), \ldots, \varphi
(\bar x, y^{(r)})),$$ is a superposition formula for the lifted
Lie-Vessiot system $\vec X^{r}$ in $M^r$. For each $\bar x \in U$ we
consider the map $$\sigma_{\bar x}\colon M \to M, \quad x\mapsto
\varphi(\bar x, x).$$ It is clear that $\sigma_{\bar x}$ is a
complex analytic automorphism of $M$. We denote by $\bar
\sigma_{\bar x}$ for the cartesian power of $\sigma_{\bar x}$ acting
component by component,
$$ \bar \sigma_{\bar x}\colon M^r \to M^r, \quad \bar y \mapsto
\bar \varphi(\bar x, \bar y).$$ The map $\bar\sigma_{\bar x}$ is a
complex analytic automorphism of $M^r$.

\begin{lemma}[(Extension Lemma)] \label{LMextensionlemma}
If $\bar x$ and $\bar y$ are in $U$ and there exist $\bar z\in M^r$
such that $\bar\varphi(\bar x,\bar z) = \bar y$, then $\bar z$ is
also in $U$.
\end{lemma}

\begin{proof}
Let us consider $\bar x$, $\bar y$, $\bar z$ in $U$ such that $\bar
\varphi(\bar x,\bar z) = \bar y$. Then we have a commutative diagram
$$\xymatrix{M \ar[rd]^-{\sigma_{\bar x}} \\ M \ar[r]^-{\sigma_{\bar y}} \ar[u]^-{\sigma_{\bar z}} & M}$$
and it is clear that $\sigma_{\bar z} = \sigma_{\bar y}\sigma_{\bar
x}^{-1}$, or equivalently
\begin{equation}\label{ATeq21}
\varphi(\bar z, \xi) = \varphi(\bar y, \sigma_{\bar x}^{-1}(\xi))
\end{equation}
for all $\xi$ in $M$. Now let us consider $\bar x$ and $\bar y$ in
$U$ and $\bar z\in M^r$ verifying the same relation $\bar
\varphi(\bar x,\bar z) = \bar y$. We can define $\varphi(\bar z,
\xi)$ as in equation \eqref{ATeq21}.


 Let $V$ be the set of all $r$-frames $\bar z$ satisfying $\varphi(\bar x, \bar z) = \bar y$
for certain $\bar x$ and $\bar y$ in $U$. Let us see that $V$ is an open subset
containing $U$. Relation $\bar \varphi(\bar x,\bar z) = \bar y$ is
equivalent to $\bar z = \bar\sigma_{\bar x}^{-1}(\bar y)$, and then
it is clear that
$$V = \bigcup_{\bar x\in U} \bar\sigma^{-1}_{\bar x}(U),$$
which is union of open subsets, and then it is an open subset. The
superposition law $\varphi$ naturally extends to $V\times M$, by
means of formula \eqref{ATeq21}. Because of the maximality of this
superposition law, we conclude that $U$ coincides with $V$ and then
such a $\bar z$ is in $U$.
\end{proof}

\begin{lemma}
Let $G$ be following set of automorphisms of $M$,
$$G = \{\sigma_{\bar x} \,|\, \bar x \in U \}.$$
Then $G$ is a group of automorphisms of $M$.
\end{lemma}

\begin{proof}
  Let us consider two elements $\sigma_1$, $\sigma_2$ of $G$.
  They are respectively of the form
  $\sigma_{\bar x}$, $\sigma_{\bar y}$ with $\bar x$ and $\bar y$ in
  $U$. $\bar \varphi$ is a superposition law for $\vec X^r$.
  From the local existence of solutions for differential equations
  we know that there exist $\bar z\in M^r$ such that
  $\bar\varphi(\bar x,\bar z) = \bar y$. From the previous lemma we
  know that $\bar z\in U$. We have that $\sigma_{\bar y}\sigma_{\bar
  x}^{-1} = \sigma_{\bar z}$. Then $\sigma_2\sigma_1^{-1}$ is in $G$
  and it is a subgroup of the group of automorphisms of $M$.
\end{proof}

  Now let us see that $G$ is endowed with an structure of complex analytic Lie group
and that the Lie algebra $\mathcal R(G,M)$ of its fundamental
fields in $M$ contains the Lie-Vessiot-Guldberg algebra $\mathfrak
g(\vec X)$ of $\vec X$. The main idea is to translate infinitesimal
deformations in $U$ to infinitesimal deformations in $G$. This
approach goes back to Vessiot \cite{Vessiot1894}.

\medskip

  Consider $\bar x$ in $U$, and a tangent vector $\vec v_x \in T_{\bar
  x}U$. We can project this vector to $M$ by means of the superposition
  principle $\varphi(\bar x, x)$. Letting $x$ as a free variable we
  define a vector field $\vec V$ in $M$,
  $$\vec V_{\varphi(\bar x,x)} = \varphi'_{(\bar x,x)}(\vec v_x );$$
  where $\varphi'_{(\bar x, x)}$ is the tangent map to $\varphi$ at
  $(\bar x, x)$:
  $$ \varphi'_{(\bar x, x)}\colon T_{(\bar x,x)}(U\times M) \to
T_{\varphi(\bar x,x)}M.$$

  This assignation is linear, and then we obtain an injective map
  $$T_{\bar x}U \to \mathfrak X(M),  \quad \vec v_x \mapsto \vec V$$
  that identifies $T_{\bar x}U$ with a finite dimensional space
  of vector fields in $M$ which we denote by $\mathfrak g_{\bar x}$. Let us see that
  this space does not depends on $\bar x$ in $U$. Consider another $r$-frame
  $\bar y\in U$. There exist a unique $\bar z\in U$ such that
  $\bar\varphi(\bar x,\bar z) = \bar y$. Consider the map,
  $$R_{\bar z}\colon U \to M^r,\quad \xi \mapsto \bar\varphi(\xi, \bar
  z), \quad R_{\bar z}(\bar x).$$
  We have $\bar \varphi(\bar y, \xi) = \bar \varphi(R_{\bar z}\bar x, \xi)$,
  and then the vector field $\vec V$ induced in $M$ by the tangent vector
  $\vec v_x\in T_{\bar x}U$ is the same vector field that the induced
  by the tangent vector $R'_{\bar z}(\vec v)\in T_{\bar y}U$. Then
  we conclude that $\mathfrak g_{\bar x} \subseteq \mathfrak g_{\bar
  y}$. We can gave the same argument for the reciprocal by taking $\bar w$ such that
  $\bar\varphi(\bar y,\bar w) = \bar x$. Then $\mathfrak g_{\bar
  x}=\mathfrak g_{\bar y}$. We denote by $\mathfrak g$ this finite dimensional space
  of vector fields in $M$ and $s$ its complex dimension.
  This space $\mathfrak g$ is a quotient of $T_{\bar x}U$, so that Lie's inequality $s\leq nr$ holds.

\medskip

  Let us see that $\mathfrak g$ is a Lie algebra. We can invert, at
least locally, the superposition law with respect to the last
component,
$$x = \varphi(\bar x,\lambda), \quad \lambda = \psi(\bar x, x).$$
Denote by $\Psi$ the sheaf generated by the components $\psi_i$ of
these local inversions; $\Psi$ is a local superposition law.
Consider $\vec v\in T_{\bar x}U$ and $\vec V$ the induced vector
field $\vec V\in \mathfrak g$. It is just an observation that $\vec
V^r_{\bar x}$, the value at $\bar x$ of the $r$-th cartesian power
of $\vec V$ also induces the same vector field $\vec V$. Then, in
the language of small displacements ,
$$x + \varepsilon \vec V_x = \varphi(\bar x + \varepsilon \vec V^r_{\bar x},\lambda), \quad \lambda =
\psi(\bar x + \varepsilon \vec V^{r}_{\bar x}, x + \varepsilon \vec
V_x).$$ Then $\vec V^{r+1}\psi_i = 0$, and $\Psi$ is a sheaf of
first integrals of $r+1$ cartesian powers of the fields of
$\mathfrak g$. By using the the same argument that in the proof of Theorem
\ref{C1THElocallie} we conclude that $\mathfrak g$
spans a finite dimensional Lie algebra.

\medskip

  Now let us consider the surjective map $\pi\colon U \to G$. For
$\sigma\in G$ the preimage $\pi^{-1}(\sigma)$ is defined by analytic
equations,
\begin{equation}\label{ATeq22}
\pi^{-1}(\sigma) = \{\bar x \in U\,\,|\,\, \forall x\in M
\,\,\varphi_i(\bar x, x) = \sigma(x)\}.
\end{equation}

Let us see that $\pi^{-1}(\sigma)$ is a closed sub-manifold of $U$. Let us
compute the tangent space to $\pi^{-1}(\sigma)$ at $\bar x$. A
tangent vector $\vec v\in T_{\bar x}U$ is tangent to the fiber of
$\sigma$ if and only if $\vec v$ is into the kernel of the canonical
map $T_{\bar x}U \to \mathfrak g$. Therefore, $\pi^{-1}(\sigma)$
has constant dimension, so that it is defined by a finite subset of
the equations \eqref{ATeq22}. Hence, the stalk $\pi^{-1}(\sigma)$ is
a closed submanifold of $U$ of dimension $nr-s$.

In such case there is a unique analytic structure on $G$ such that
$\pi\colon U \to G$ is a fiber bundle. Consider $\mathfrak g^r$ the
Lie algebra spanned by cartesian power vector fields,
$$\vec V^r = \vec V^{(1)} + \ldots + \vec V^{(r)},$$
with $\vec V \in \mathfrak g$. This Lie algebra $\mathfrak g^r$ is
canonically isomorphic to $\mathfrak g$. The fields of $\mathfrak g^r$ are
projectable by $\pi$. The Lie algebra $\mathfrak g$ is then
identified with the Lie algebra $\mathcal R(G)$ of right invariant vector fields
in $G$. We deduce that $\mathfrak g$ is the algebra $\mathcal R(G,M)$ of fundamental
fields of $G$ in $M$.

Finally, let us see that the elements of $\mathcal R(G,M)$ span the
Lie-Vessiot-Guldberg algebra of $\vec X$. For all $t_0\in S$ we
have,
$$\varphi(\bar x + \varepsilon (\vec X^r_{t_0})_{\bar x}, \lambda) = x +
(\vec X_{t_0})_x,$$ and then $\vec X_{t_0}$ is the vector field of
$\mathfrak g$ induced by the tangent vector $(\vec X^r_{t_0})_{\bar
x}$ at any $\bar x$ in $U$. Then $\vec X_{t}\in \mathfrak g$ for all
$t\in S$ and therefore there are $\vec V_i\in\mathfrak g$ and
analytic functions $f_i(t)\in \mathcal O_S$ for $i=1,\ldots, s$ such
that,
$$\vec X = \partial + \sum_{i=1}^s f_i(t)\vec V_i,$$
and $\vec X$ is a Lie-Vessiot system in $M$ related to the action of
$G$.

\begin{lemma}\label{lema6.7}
The action of $G$ on $M$ is pretransitive.
\end{lemma}

\begin{proof}
First, by Lemma \ref{LMextensionlemma} the open subset $U$ is union
of principal orbits. Let us prove that the space of orbits $U/G$ in
a complex analytic manifold. Consider $\pi\colon U \to G$, $\bar
x\mapsto \sigma_{\bar x}$ as in the previous lemma. Let $U_0$ be
preimage of $Id$, which is a closed submanifold of $U$. Consider the
map,
$$\pi_2\colon U\to U, \quad \bar x \mapsto \sigma_{\bar
x}^{-1}\bar x,$$ then $\pi(\pi_2(\bar x)) = Id$, and the image of
$\pi_2$ is $U_0$. The two projections  $\pi\colon U \to G$ and
$\pi_2\colon U\to U_0$ give a decomposition $U = U_0\times G$ and
then the quotient $U/G\simeq U_0$ is a complex analytic manifold. We
conclude that the action of $G$ on $M$ is pretransitive.
\end{proof}

We conclude that $\vec X$ is a Lie-Vessiot system associated to the
pretransitive action of $G$ on $M$, this ends the proof of theorem
\ref{T1}.

\begin{example} \label{ExampleWeierstrass}
Let us consider Weierstrass equation for the $\wp$ function,
$$\dot\wp^2 = 4(\wp^3+g_2\wp+g_3),$$
the classical addition formula for Weierstrass equation,
$$\wp(a+b) = - \wp(a) - \wp(b) - \frac{1}{4}\frac{\dot\wp(a)-\dot\wp(b)}{\wp(a)-\wp(b)}$$
can be understood as a superposition law for the differential equation,
\begin{equation}\label{nonautonomousWeierstrass}
\dot x = 2\sqrt{f(t)(x^ 3+g_2x+g_3)}.
\end{equation}
The general solution of \ref{nonautonomousWeierstrass} is:
\begin{equation}\label{SolutionW}
x = \wp\left(\int_{t_0}^xf(\xi)d\xi + \lambda \right).
\end{equation}
By taking off the constant $\lambda$ of the formula \ref{SolutionW} we obtain the metioned
superposition law,
$$\varphi(x,\lambda) = -x -\lambda - \frac{1}{2}\frac{\sqrt{x^3-g_2x-g_3}-\sqrt{\lambda^3-g_2\lambda-g_3}}{x-\lambda}.$$
Where the square roots are defined in a doble-sheet ramified covering $\mathbb C$. 
If we apply Vessiot global method here, it is clear that we recover the elliptic curve group structure.
However, infinitesimal Lie approach gives us no information about
this non-linear algebraic group, but just about its Lie algebra. Using Lie's infinitesimal approach,
equation \ref{nonautonomousWeierstrass} is just equivalent to,
$$\dot x = f(t).$$
This example can be generalized to addition formulas of abelian functions in several variables.
In those cases, non-linear and non-globally linearizable superposition laws appear.
\end{example}

\section{Rational Superposition Laws}

  Using this global philosophy it is easy to jump into the
algebraic category. First, is follows easly that rational
superposition laws lead to rational actions of algebraic groups.
From now on let $M$ be an algebraic variety. 

\begin{theorem}\label{T2}
Let $\vec X$ be a non-autonomous meromorphic vector field
in an algebraic manifold $M$ that admits a rational superposition law. Then,
it is a Lie-Vessiot system related to a algebraic action of an algebraic
group if $M$.
\end{theorem}

\begin{proof}
  First, note that the extension lemma \ref{LMextensionlemma} also works in the algebraic case.
So that we can assume that the superposition law
$$\varphi \colon U \times M \to M$$
is defined in a Zariski open subset $U\subset M^r$ such that, if $\bar x, \bar y$
are in $U$ and there exist $\bar z$ such that $\bar \varphi(\bar x, \bar z) = \bar y$
then $\bar z$ is also in $U$.
Then, by Lemma \ref{lema6.7} we have a decomposition,
$U = G \times U_0$ being $U_0$ the fiber of the identity in $U$. Let us see that
we can indentify $G$ with an algebraic submanifold of $U$. Let us fix $\bar\lambda\in U_0$,
then $G$ is identified with the set $\varphi(\bar x, \bar \lambda)$ where $x$ varies in $U$,
which is the image of the map:
$$\psi\colon U \to U, \quad \psi(\bar x) = \varphi(\bar x, \bar \lambda),$$
This map $\psi$ is clearly rational. The image of a rational map is, by \emph{Weil's projection
theorem}, boolean combination of algebraic subsets of $U$. But, in the other hand,
it has been proved that this image is analiticaly isomorphic to $G$. Then, it follows
that it embedding of $G$ into $U$ whose image is an algebraic submanifold of $U$. Then,
$G$ inherits this algebraic structure. It is clear that the composition law and the action are then
algebraic.
\end{proof}

\section{Strongly Normal Extensions}

From now on, let $\vec X$ be a non-autonomous \emph{meromorphic} vector field
in an algebraic manifold $M$ with coefficients in the Riemann surface 
$S$ that admits a rational superposition law. Applying Theorem \ref{T2}, we
also  consider the algebraic group $G$ that realizes $\vec X$ as 
an algebraic Lie-Vessiot system.

The objective of this section is to prove that a rational differential equation
with meromorphic coefficients in a Riemann surface $S$, admitting a rational superposition law,
have its general solution in a differential field extension of the field
of meromorphic functions in $S$ which is an \emph{strongly normal extension}
in the sense of Kolchin. Then, superposition laws fall into the differential
Galois theory developed in \cite{Kolchin1973}.

  The field $\mathcal M(S)$ of meromorphic functions in $S$ is a differential
field, endowed of the derivation $\partial$, its field of constants if clearly
$\mathbb C$. For a differential field $K$ we denote its field of constants by $C_K$. 
Let $K$, $L$ be two differential fields. By a differential field
morphism we mean a field morphism $\varphi\colon K \to L$ which commutes with
the derivation. When we fix a non-trivial (and then injective) differential field morphism,
$K\subset L$ we speak of a differential field extension. By a $K$-isomorphism of $L$
we mean a differential field morphism from $L$ into a differential field extension
of $K$ that fix $K$ pointwise. We will denote by ${\rm Aut}_K(L)$ the set of differential 
automorphisms of $L$ that fix $K$ pointwise.

 The following lema tell us that the vector field $X$ can not have movable
singularities.

\begin{lemma}
There is a discrete set $R$ in $S$ such that $\vec X$ is holomorphic in $M\times S^\times$
being $S^\times = S\setminus R$. 
$$\vec X = \delta + \sum_{i=1}^sf_i(t)\vec X_i,$$
being the $f_i(t)$ meromorphic functions in $S$ with poles in $R$. 
\end{lemma}

\begin{proof}
Let $t_0$ be a point of $S$ such that $\vec X$ is defined at least at one point of the fiber $t=t_0$
in $M\times S$. Then, it is defined in at least an open subset of the fiber $t=t_0$. It implies
that there exist then $r$ local solutions $x_1(t), \ldots, x_r(t)$, defined in an in $S'\subset S$ around $t_0$ 
that form a fundamental system of solutions. Then, by the superposition law, its follows that for any particular
initial condition $x_0\in M$ there exist $\lambda\in M$ such that $\varphi(x_1(t),\ldots,x_r(t),\lambda)$ is
the solution of $\vec X$ with initial condition $x(t_0) = x_0$. It follows that the fiber $t=t_0$ does not
meet the singular locus of $\vec X$.

Thus, the singular locus of $\vec X$ should be an union of fibers of the proyection $S\times M \to S$. Those
fibers should be isolated, because the vector field $\vec X$ is meromorphic. 
\end{proof}

  From now on we consider $S^\times$ as above, and we set $\widetilde{S^\times}$ to be the
universal covering of $S^\times$.

\begin{lemma}
Any germ of solution $x(t)$ defined in
$S'\subset  S^\times$ can be prolonged into a solution defined in the universal
covering $\widetilde{S^\times}$ of $S^\times$.
\end{lemma}

\begin{proof}
It is enough to see that any solution can be analitically prolongued along 
any diffentiable path. Let $\gamma\colon [0,1]\to S^\times$ be a differentiable
path. For any point $\bar t$ in the image of $\gamma$ we consider a simply connected
open neighborhood $U_{\bar t}$ such that there exist a fundamental system of solutions 
$x_1^{\bar t},\ldots x_r^{\bar t}$ defined in $U_{\bar t}$. Now, we apply that the
interval $[0,1]$ is compact, so we can take just a finite number of open subsets
$U_1,\ldots, U_n$. In each one of this open subset the general solution exist for
any initial condition. So that, we can continue any solution defined around $\gamma(0)$
along $\gamma$ to $\gamma(1)$. 
\end{proof}

Let $x(t)$ be a particular solution of $\vec X$ defined in $\widetilde{S^\times}$. Then,
we can consider $\mathcal M(S)\subset \mathcal M(S)\langle x(t) \rangle$, the differential field
extension spanned by the coordinates of $x(t)$ in any affine chart. 

\begin{lemma}\label{uniqueGal}
Let $ \bar x(t) = (x_1(t),\ldots,x_r(t))$ be a fundamental system of solutions of $\vec X$, and $\sigma(t)$ a solution
of the associated automorphic system $\vec A$. Then  $\mathcal M(S)\langle \bar x(t)\rangle = \mathcal M(S)\langle\sigma(t)\rangle$.
\end{lemma}

\begin{proof}
First, there exist a $r$-uple $\bar \lambda$ in $U$ such that $\bar x(t) = \sigma(t)\cdot \bar \lambda$. 
Then, it is clear that the coordinates of $\bar x(t)$ are rational functions on the coordinaes of $\sigma(t)$, 
i. e. $\mathcal M(S)\langle \bar x(t)\rangle  \subset 
\mathcal M(S)\langle\sigma(t)\rangle$.
The converse follow from the fact that there is a surjective morphism $\pi\colon U \to G$ that sends $\bar x(t)$ to
a solution $\tau(t)$ of $\vec A$. Then the coordinates of $\tau(t)$ are, by this proyection, rational
funcions of the coordinates of $\bar x(t)$ and we have that $\mathcal M(S)\langle \tau(t)\rangle  \subset 
\mathcal M(S)\langle x(t) \rangle$. Finally, by Proposition \ref{automorphicproperties} (iii), we know
that $\tau(t)$ and $\sigma(t)$ are related by a left translation in $G$, so that $\mathcal M(S)\langle \sigma(t) \rangle$
coincides witn $\mathcal M(S)\langle \tau(t) \rangle$. 
\end{proof}

By the above Lemma \ref{uniqueGal} any fundamental system of solutions of $\vec X$, or equivalenty
any solution of $\vec A$ span the same differential field extension of $\mathcal M(S)$. Then, we
can define:

\begin{Definition}
We call the Galois differential field of $\vec X$ to the differential
field $\mathcal M(S)\langle \bar x(t)\rangle$
spanned by the coordinates of a fundamental system of solutions $\bar x(t)$, defined in $\widetilde{S^\times}$. 
\end{Definition}

From now on we denote by $L$ to the Galois differential field of $\vec X$. We have 
that $\mathcal M(S) \subset L \subset \mathcal M(\widetilde{S^\times})$. It is clear
that if $x(t)$ is a particular solution of $\vec X$ then $\mathcal M(S)\langle x(t)\rangle$
is contained in $L$.

\begin{Definition}
  Let $K$ be a differential field whose field of constants if $\mathbb C$. A strongly 
  normal extension is a differential field extension $K\subset L$ such that,
  \begin{itemize}
  \item[(i)] the field of constants of $L$ is $\mathbb C$,
  \item[(ii)] for any $K$-isomorphism $\sigma\colon L \to \mathcal U$ into a differential
  field extension $\mathcal U$ of $L$ we have that $L\cdot \sigma L \subset L\cdot C_{\mathcal U}$.
  \end{itemize}
\end{Definition}

\begin{theorem}\label{T3}
Let $\vec X$ be a non-autonomous meromorphic vector field in an algebraic manifold $M$
with coefficients in the Riemann surface $S$ that admits a rational superposition law.
Let $L$ be the Galois differenial field of $\vec X$. Then,
$\mathcal M(S) \subset L$ is a strongly normal extension in the sense of Kolchin.
Moreover, 
each particular solution $\sigma(t)$ of the associated automorphic system induce an
injective map,
$${\rm Aut}_{\mathcal M(S)}(L) \to G.$$
which is an anti-morphism of groups.
\end{theorem}

\begin{proof}
We consider, as above, $\sigma(t)$ a solution of the associated automorphic system,
so that $L = \mathcal M(S)\langle \sigma(t) \rangle.$ 
Let $F\colon L \to \mathcal U$ be a $\mathcal M(S)$-isomorphism. Then,
acting coordinate by coordinate $F(\sigma(t))$ is an element of $G$ with coordinates
in $\mathcal U$. Let us recall $F$ fixes $\mathcal M(S)$ pointwise, we have that
the coordinates $F(\sigma(t))$ satisfy all the differential equations satisfied by
the coordinates of $\sigma(t)$ with coefficients in $\mathcal M(S)$. Then $F(\sigma(t))$
is a solution of $\vec A$. Let us denote $F(\sigma(t)) = \tau(t)$. 

By the properties of automorphic systems, we have
that $\sigma(t)^{-1}\cdot \tau(t)$ which is an element of $G$ with coordinates in $\mathcal U$
is constant, that is, its coordinates are in $C_\mathcal U$. Let us call it $\lambda =  
\sigma(t)^{-1}\cdot \tau(t)$. Then, we have $L \cdot \sigma L \subset L(\lambda) \subset L \cdot C_{\mathcal U}$,
and we conclude that $\mathcal M(S)\subset L$ is a strongly normal extension.

Let us construct the anti-morphism of groups above stated. We fix the solution $\sigma(t)$
of $\vec A$. Let $F$ be a $K$-automorphism of $L$. By the above argument $F(\sigma(t))$ is also a solution of $\vec A$,
so that there exist a unique $\lambda_F$ such that $F(\sigma(t)) = \sigma(t) \cdot \lambda_F$. Then we define:
$$\Phi_\sigma \colon {\rm Aut}_{\mathcal M(S)}(L) \to G, \quad F \to \lambda_F.$$
If we consider a pair of automorphisms $F,$ and $G,$ we have:
$$F(G(\sigma(t)) = F(\sigma(t))\cdot \lambda_G) = \sigma(t)\cdot \lambda_F \lambda_G,$$
so that it is clear that it is anti-morphism. Finally, for proving that this map is
injective it suffices to recall that $\sigma(t)$ spans $L$. Thus, if an automorphism
let $\sigma(t)$ fixed then it is the identity. 
\end{proof}

\begin{remark}
  This problem was also apprached in a purely algebraic way,
with no explicit relation with superposition laws. See,
for instance \cite{BlazquezSCharris} Theorem 4.40.,
or \cite{Kolchin1973} Chapter VI, Section 7.
 
\end{remark}

\section*{Acknowledgements}

  This research of both authors has been partially financed
by MCyT-FEDER Grant MTM2006-00478 of spanish goverment. The
first author is also supported by {\sc Civilizar}, the research 
agency of Universidad Sergio Arboleda. We also acknowledge 
prof. J.-P. Ramis and prof. E. Paul for their support 
during the visit of the first author to Laboratoire Emile Picard. 
We are also in debt with J. Mu\~noz of Universidad de Salamanca for 
his help with the original work of S. Lie. We thank also P. Acosta,
T. Lazaro and C. Pantazi who shared with us
the seminar of algebraic methods in differential equations in Barcelona. 
We also want to recognize our debt with J. Kovacic, who recently passed away,
and shared his knowledge with us during his visit to Barcelona.

\bigskip

{\sc\noindent David Bl\'azquez-Sanz \\
Escuela de Matem\'aticas\\
Universidad Sergio Arboleda\\
Calle 74, no. 14-14 \\
Bogot\'a, Colombia\\
}
E-mail: {\tt david.blazquez-sanz@usa.edu.co}

\bigskip

{\sc\noindent Juan Jos\'e Morales-Ruiz \\
Departamento de Inform\'atica y Matem\'aticas\\
Escuela de Caminos Canales y Puertos\\
Universidad Polit\'ecnica de Madrid
Madrid, Espa\~na\\
}
E-mail: {\tt juan.morales-ruiz@upm.es}

\end{document}